\documentclass{amsproc}
\usepackage{epsfig, a4wide}
\usepackage[latin1]{inputenc}
\usepackage{amsthm, amssymb}
\usepackage{amsfonts, amscd}
\usepackage{multicol}
\newtheorem{theorem}{Theorem}[section]
\newtheorem{lemma}[theorem]{Lemma}

\theoremstyle{definition}

\newtheorem{example}[theorem]{Example}

\theoremstyle{remark}
\newtheorem{remark}[theorem]{Remark}
\newtheorem{proposition}[theorem]{Proposition}
\newtheorem{corollary}[theorem]{Corollary}

\newtheorem{problem}{Problem}

\numberwithin{equation}{section}
\usepackage[all]{xy}

\newcommand{\Z}{{\mathbb{Z}}}

\newcommand{\R}{{\mathbb{R}}}
\newcommand{\N}{{\mathbb{N}}}

\begin{document}

\title{Projective bundles over small covers and topological triviality problem}

\author{Shintar\^o KUROKI}
\address{Osaka City University Advanced Mathematical Institute,
3-3-138 Sugimoto, Sumiyoshi-ku Osaka 558-8585, JAPAN}
\email{kuroki@scisv.sci.osaka-cu.ac.jp}
\thanks{
The first author was supported in part by the JSPS Strategic Young Researcher Overseas Visits Program for Accelerating Brain Circulation
"Deepening and Evolution of Mathematics and Physics, Building of International Network Hub based on OCAMI".
The second author was supported in part by  grants from NSFC (No. 11371093 and No. 10931005) and and RFDP (No. 20100071110001).}
\author{Zhi L\"U}
\address{School of Mathematical Science, Fudan University, Shanghai, 200433, P.R.\ China}
\email{zlu@fudan.edu.cn}

%    General info
\subjclass[2010]{}


\begin{abstract}
This paper investigates the projectivization of real vector bundles over small covers.
We first give a necessary and sufficient condition for such a projectivization to be a small cover. Then associated with moment-angle manifolds, we further study the structure of such a projectivization as a small cover. As an application, we characterize the real projective bundles over $2$-dimensional small covers by interpreting the fibre sum operation to some combinatorial operation.
Finally, we study when the projectivization of Whitney sum of the tautological line bundle
and the tangent bundle over real projective space is diffeomorphic to the product of two real projective spaces.
\end{abstract}

\maketitle
\tableofcontents

%%%%%%%%%%%%%%%%%%%%%%%%%%%%%%%%%%%%%%%%%%%%%%%%%%%%%%%%%%%%%%%%%%%%%%%%%%%
%Section 1
%%%%%%%%%%%%%%%%%%%%%%%%%%%%%%%%%%%%%%%%%%%%%%%%%%%%%%%%%%%%%%%%%%%%%%%%%%%
\section{Introduction}
\label{sect1}

A real projectivization $P(\xi)$ over a closed manifold $M$ is defined by a vector bundle $\xi$ over $M$ via antipodal maps on fibers of associated sphere bundle $S(\xi)$ (we also call a {\it real projective bundle} over $M$ in this paper).
In \cite{Borel-H},  Borel and Hirzebruch computed the total Stiefel-Whitney class of $P(\xi)$, which has been applied to the study of the equivariant cobordism (see \cite{CF}).
In his paper~\cite{St}, Stong introduced a special kind of real projective bundles (i.e., {\it Stong manifolds}, also see Section \ref{sect4}),
which can be used as generators in the Thom unoriented cobordism ring $\mathfrak{N}_*$.

As the topological version of real toric manifolds, Davis and Januszkiewicz introduced and studied a class of particularly nicely behaving manifolds $M^n$ (called {\it small covers}), each of which admits a locally standard ${\Bbb Z}_2^n$-action such that its orbit space is a
simple convex $n$-polytope $P^n$, where ${\Bbb Z}_2^{n}=\{-1,1\}^n$ is a real torus.
This establishes a direct connection among topology,
combinatorics and commutative algebra etc.
In this paper, we first consider the following natural questions:

\begin{problem}\label{pro}
Let $P(\xi)$ be a real projective bundle over a small cover.
When is also $P(\xi)$ a small cover? If so, how can we characterize its topology?
\end{problem}

With respect to Problem~\ref{pro}, we have
\begin{theorem}
\label{main1}
Let $P(\xi)$ be a real projective bundle over a small cover. $P(\xi)$  is a small cover if and only if the equivariant vector bundle $\xi$ decomposes into the Whitney sum of equivariant line bundles, i.e.,
$\xi\equiv\gamma_{1}\oplus\cdots\oplus\gamma_{k-1}\oplus\gamma_{k}$.
\end{theorem}

By this theorem, we have the following two corollaries (also see Section \ref{sect4} the definitions of generalized real Bott manifold and Stong manifold):

\begin{corollary}
Each generalized real Bott manifold is a small cover. In particular, each Stong manifold is a small cover.
\end{corollary}

The generalized Bott manifold is defined in \cite{CMS} as a special class of toric manifolds.
The generalized real Bott manifold is its real analogue.

\begin{corollary} \label{br}
Each class of $\mathfrak{N}_*$ contains a small cover as its representative.
\end{corollary}
The fact of Corollary~\ref{br} has been proved in~\cite{BR} with a different argument, and in addition, the fact
that
each class of complex cobordism contains a quasitoric manifold as its representative
was also proved in \cite{BR}. For the equivariant case, see~\cite{L, LT}. 

Associated with moment-angle manifolds, we further study the structure of a real projective bundle  $P(\xi)$ as a small cover.
As an application, we characterize the real projective bundles (as small covers) over $2$-dimensional small covers.
Our result is stated as follows:
\begin{theorem}
\label{main2}
Let $P(\xi)$ be a real projective bundle over $2$-dimensional small cover $M^{2}$ with its fibre $\R P^{k-1}$.
Then, $P(\xi)$ can be constructed from real projective bundles $P(\kappa)$ over $\R P^{2}$
and $P(\zeta)$ over $T^{2}$ by using the fibre sum $\sharp^{\Delta^{k-1}}$.
\end{theorem}

%\begin{remark} In the paper \cite{M}, Masuda gives a counter example of the cohomological rigidity of small covers,
%that is, the cohomology ring does not distinguish small covers.
%He classified homeomorphism types of height $2$ (generalized) real Bott manifolds by using the KO-ring structure of real projective spaces, where
%the {\it height $2$ (generalized) real Bott manifold} is the projective bundle of the Whitney sum of line bundles over real projective space,
%and he found counter examples of the cohomological rigidity in height $2$ (generalized) real Bott manifolds.
%Due to the Masuda's method, if we know the KO-ring of small covers then we might be able to classify
%the homeomorphism types of such projective bundles; however, unfortunately
%it is difficult to compute KO-ring structures of small covers in general.
%So, in this paper, we try to study homeomorphism types of such projective bundles from the different direction.
%Our method is closed to the Orlik-Raymond's method in \cite{OR}.
%In \cite{OR}, Orlik and Raymond show that $4$-dimensional torus manifolds can be constructed from basic torus manifolds by using connected sums.
%In fact, the small cover (resp. $2$-torus manifold) $M^{n}$ is the real version of the quasitoric manifold (resp. torus manifold) $M^{2n}$.
%Therefore, the Orlik-Raymond's method in \cite{OR} can be also applied for $2$-dimensional small covers.
%Indeed, here we can generalize their method to the projective bundles over $2$-dimensional small covers
%\end{remark}

If we take $k=1$ in the above theorem, then this gives the standard equivariant connected sums of $2$-dimensional small covers.
So by Theorem \ref{main2} we also have the following well-known facts:
\begin{corollary}
Let $M$ be a $2$-dimensional small cover.
Then $M$ can be constructed from $\R P^{2}$
and $T^{2}$ by using the equivariant connected sum $\sharp^{\Delta^{k-1}}$.
\end{corollary}

Finally, in this paper, we also prove the following theorem:
\begin{theorem}
\label{main3}
Let $\gamma$ be the tautological line bundle and $\tau_{\R P^{n}}$ be the tangent bundle over $\R P^{n}$.
Then, the following two statements are equivalent:
\begin{enumerate}
\item $P(\gamma\oplus\tau_{\R P^{n}})$ is diffeomorphic to $\R P^{n}\times \R P^{n}$;
\item $n=0,2,6$.
\end{enumerate}
\end{theorem}

The organization of this paper is as follows.
In Sections \ref{sect2} and \ref{sect3}, we recall the basic facts about small covers and projective bundles.
In Section \ref{sect4}, we give a proof of Theorem \ref{main1}, and we also give the following two characterizations of projective bundles of small covers:
(1) the characterization by the twisted product with real moment-angle manifolds;
(2) the combinatorial characterization using simple convex polytopes and some function, like Davis-Januszkiewicz's small cover.
In particular, to do (2), we introduce a new characteristic function on simple convex polytopes, called {\it projective characteristic functions}.
In Section \ref{sect5},
we prove Theorem \ref{main2}.
To do this, we use the characterization (2) and introduce a new combinatorial operation which is the combinatorial analogue of the fibre sum, called a {\it projective fibre sum}.
In Section \ref{sect6}, we classify all topological types of projective bundles over $\R P^{2}$ and $T^{2}$.
In Section \ref{sect7}, we prove Theorem \ref{main3} and propose a question which we call {\it topological triviality problem}.
This problem is motivated by the question asked by Richard Montgomery.
This paper gives the detailed proof for the results stated in \cite{Kuroki2} and also adds some results about the topological triviality problem.

%%%%%%%%%%%%%%%%%%%%%%%%%%%%%%%%%%%%%%%%%%%%%%%%%%%%%%%%%%%%%%%%%%%%%%%%%%%
%Section 2
%%%%%%%%%%%%%%%%%%%%%%%%%%%%%%%%%%%%%%%%%%%%%%%%%%%%%%%%%%%%%%%%%%%%%%%%%%%
\section{Basic properties of small cover}
\label{sect2}

In this section, we recall the notion of a small cover and the basic facts of its equivariant cohomology.

%%%%%%%%%%%%%%%%%%%%%%%%%%%%%%%%%%%%%%%%%%%%%%%%%%%%%%%%%%%%%%%%%%%%%%%%%%%
%Section 2.1
%%%%%%%%%%%%%%%%%%%%%%%%%%%%%%%%%%%%%%%%%%%%%%%%%%%%%%%%%%%%%%%%%%%%%%%%%%%
\subsection{Definition of small covers}
\label{sect2.1}
Let $M^{n}$ be an $n$-dimensional smooth closed manifold, and
$P^{n}$ a {\it simple} convex $n$-polytope, i.e., precisely $n$ {\it facets} (codimension-$1$ faces) of $P^{n}$ meet at each vertex.
Put $\Z_{2}=\{-1,1\}$.
We call $M^{n}$  a {\it small cover} if $M$ admits a $\Z_{2}^{n}$-action such that
\begin{description}
\item[(a)] the $\Z_{2}^{n}$-action is {\it locally standard}, i.e., locally the same as the standard $\Z_{2}^{n}$-action on $\R^{n}$, and
\item[(b)] its orbit space has the structure of a simple convex polytope $P^{n}$, i.e.,
the corresponding orbit projection map
$\pi:M^{n}\to P^{n}$ is constant on $\Z_{2}^{n}$-orbits
and maps every rank $k$ orbit (i.e., every orbit isomorphic to $\Z_{2}^{k}$) to an interior point of
a $k$-dimensional face of the polytope $P^{n}$, $k=0, 1,   \ldots,\ n$.
\end{description}
  It is easy to see that $\pi$ sends $\Z_{2}^{n}$-fixed points in $M^{n}$ to vertices of $P^{n}$ by using the above condition (b).
We often call $P^{n}$ an {\it orbit polytope} of $M$.

%%%%%%%%%%%%%%%%%%%%%%%%%%%%%%%%%%%%%%%%%%%%%%%%%%%%%%%%%%%%%%%%%%%%%%%%%%%
%Section 2.2
%%%%%%%%%%%%%%%%%%%%%%%%%%%%%%%%%%%%%%%%%%%%%%%%%%%%%%%%%%%%%%%%%%%%%%%%%%%
\subsection{Construction of small covers}
\label{sect2.2}

Conversely, for a given simple polytope $P^{n}$, the small cover $M^{n}$ with orbit projection $\pi:M^{n}\to P^{n}$
can be reconstructed by using the {\it characteristic function} $\lambda:\mathcal{F}\to (\Z/2\Z)^{n}$,
where $\mathcal{F}$ is the set of all facets in $P$ and $\Z/2\Z=\{0,1\}$.
In this subsection, we recall this construction (see \cite{BP, DJ} for details).

Following the definition of a small cover $\pi:M\to P$, we have that $\pi^{-1}({\rm int}(F^{n-1}))$ consists of $(n-1)$-rank orbits,
in other words, the isotropy subgroup at $x\in \pi^{-1}({\rm int}(F^{n-1}))$
is $K\subset \Z_{2}^{n}$ such that $K\simeq \Z_{2}$,
where ${\rm int} (F^{n-1})$ is the relative interior of the facet $F^{n-1}$.
Hence, the isotropy subgroup at $x$ is determined by a primitive vector $v\in (\Z/2\Z)^{n}$
 such that $(\textbf{-1})^{v}$ generates the subgroup $K$,
where $(\textbf{-1})^{v}=((-1)^{v_{1}},\ldots,(-1)^{v_{n}})$ for $v=(v_{1},\ldots,v_{n})\in (\Z/2\Z)^{n}$.
In this way, we obtain a function $\lambda$ from the set of facets of $P^{n}$, denoted by $\mathcal{F}$, to vectors in $(\Z/2\Z)^{n}$.
We call such $\lambda:\mathcal{F}\to (\Z/2\Z)^{n}$ a {\it characteristic function} or a {\it coloring} on $P^{n}$.
We often describe $\lambda$ as the $(m\times n)$-matrix $\Lambda=(\lambda(F_{1})\cdots \lambda(F_{m}))$ for $\mathcal{F}=\{F_{1},\ldots,F_{m}\}$ with a given ordering,
and we call this matrix a {\it characteristic matrix}.
Since the $\Z_{2}^{n}$-action is locally standard,
a characteristic function has the following property (called {\it the property $(\star)$}):
\begin{description}
\item[$(\star)$] if $F_{i_1}\cap\cdots\cap F_{i_n}\not=\emptyset$ for $F_{i_j}\in \mathcal{F}$ ($j=1,\ \ldots,\ n$),
then $\{\lambda(F_{i_1}), \ldots, \lambda(F_{i_n})\}$ spans $(\Z/2\Z)^{n}$.
\end{description}

 An interesting thing is that one can also construct small covers by using a given $n$-dimensional simple convex polytope $P$ and a characteristic function $\lambda$ with the property $(\star)$.
Let $P$ be an $n$-dimensional simple convex polytope.
Suppose that a characteristic function $\lambda:\mathcal{F}\to (\Z/2\Z)^{n}$ with the above property $(\star)$ is defined on $P$.
Small covers can be constructed from $P$ and $\lambda$ as the quotient space
$\Z_{2}^{n}\times P/\sim_{\lambda}$,
where the equivalence relation $\sim_{\lambda}$ on $\Z_{2}^{n}\times P$ is defined as follows:
$(t, x)\sim_{\lambda} (t', y)$ if and only if $x=y\in P$ and
\begin{eqnarray*}
\begin{array}{ll}
t=t' & {\rm if}\ x\in {\rm int} (P); \\
t^{-1}t'\in \langle (\textbf{-1})^{\lambda(F_{i_1})}, \cdots, (\textbf{-1})^{\lambda(F_{i_r})} \rangle\simeq \Z_{2}^{r} & {\rm if}\ x\in {\rm int} (F_{i_1}\cap\cdots \cap F_{i_r}),
\end{array}
\end{eqnarray*}
where $\langle(\textbf{-1})^{\lambda(F_{i_1})}, \cdots, (\textbf{-1})^{\lambda(F_{i_r})}\rangle\subset \Z_{2}^{n}$
denotes the subgroup generated by $(\textbf{-1})^{\lambda(F_{i_j})}$ for $j=1,\ \ldots,\ r$ with $r\leq n$.
The small cover $\Z_{2}^{n}\times P/\sim_{\lambda}$ defined by this way is usually denoted  by $M(P, \lambda)$.

Summing up, we have the following relations:
\begin{center}
\fbox{
\begin{tabular}{c}
Small covers \\
with $\Z_{2}^{n}$-actions
\end{tabular}
}
$
\begin{array}{c}
\longrightarrow  \\
\longleftarrow
\end{array}
$
\fbox{
\begin{tabular}{c}
Simple convex polytopes \\
with characteristic functions
\end{tabular}
}
\end{center}

%%%%%%%%%%%%%%%%%%%%%%%%%%%%%%%%%%%%%%%%%%%%%%%%%%%%%%%%%%%%%%%%%%%%%%%%%%%
%Section 2.3
%%%%%%%%%%%%%%%%%%%%%%%%%%%%%%%%%%%%%%%%%%%%%%%%%%%%%%%%%%%%%%%%%%%%%%%%%%%
\subsection{Equivariant cohomology and ordinary cohomology of small cover}
\label{sect2.3}

In this subsection, we
recall the equivariant cohomology and ordinary cohomology of the small covers (see \cite{BP, DJ} for details).
Let $M=M(P,\lambda)$ be an $n$-dimensional small cover.
We denote an ordered set of  facets of $P$ by $\mathcal{F}=\{F_{1},\ldots,F_{m}\}$ such that $\cap_{i=1}^{n}F_{i}\not=\emptyset$.
Then, we may take the characteristic functions on $F_{1},\ldots,F_{n}$ as
\begin{eqnarray*}
\lambda(F_{i})=e_{i}
\end{eqnarray*}
where $e_{1},\ldots,e_{n}$ are the standard basis vectors of $(\Z/2\Z)^{n}$.
That is, we can write the characteristic matrix as
\begin{eqnarray*}
\Lambda=(I_{n}\ |\ \Lambda'),
\end{eqnarray*}
where $I_{n}$ is the $(n\times n)$-identity matrix and $\Lambda'$ is an $(l\times n)$-matrix, where $l=m-n$.

The {\it equivariant cohomology} of a $G$-manifold $X$ is defined by the ordinary cohomology of $EG\times_{G}X$, where $EG$ is the total space of a universal $G$-bundle, and denoted by $H_{G}^{*}(X)$.
In this paper, we assume the coefficient group of cohomology is $\Z/2\Z$.
Due to \cite{DJ}, the ring structure of the equivariant cohomology of a small cover $M$ is given by the following formula:
\begin{eqnarray*}
H_{\Z_{2}^{n}}^{*}(M)\simeq  \Z/2\Z[\tau_{1},\ldots,\tau_{m}]/\mathcal{I},
\end{eqnarray*}
where the symbol $\Z/2\Z[\tau_{1},\ldots,\tau_{m}]$ represents the polynomial ring generated by the degree $1$ elements $\tau_{i}$ $(i=1,\ldots,m)$,
and the ideal $\mathcal{I}$ is generated by the following monomial elements:
\begin{eqnarray*}
\prod_{i\in I}\tau_{i}
\end{eqnarray*}
where $I$ runs through every subset of $\{1,\ldots,m\}$ such that $\cap_{i\in I}F_{i}=\emptyset$.
On the other hand, the ordinary cohomology ring of $M$ is given by
\begin{eqnarray*}
H^{*}(M)\simeq H_{\Z_{2}^{n}}^{*}(M)/\mathcal{J},
\end{eqnarray*}
where the ideal $\mathcal{J}$ is generated by the following degree $1$ homogeneous elements:
\begin{eqnarray*}
\tau_{i}+\lambda_{i1}x_{1}+\cdots \lambda_{i l}x_{l},
\end{eqnarray*}
for $i=1,\ldots,n$.
Here, $(\lambda_{i1}\cdots \lambda_{i l})$ is the $i$th row vector of $\Lambda'$ ($i=1,\ldots,n$), and
$x_{j}=\tau_{n+j}$ ($j=1,\ldots, l$).

Note that the above ideal $\mathcal{J}$ coincides with the ideal generated by $\pi^{*}(H^{+}(B\Z_{2}^{n}))={\rm Im}\ \pi^{+}$, i.e.,
\begin{eqnarray*}
\mathcal{J}=\langle {\rm Im}\ \pi^{+} \rangle,
\end{eqnarray*}
where $H^{+}(B\Z_{2}^{n})=H^{*}(B\Z_{2}^{n})\setminus H^{0}(B\Z_{2}^{n})$ and $\pi^{*}:H^{*}(B\Z_{2}^{n})\to H^{*}_{\Z_{2}^{n}}(M)$ is the induced homomorphism from
the natural projection $E\Z_{2}^{n}\times_{\Z_{2}^{n}}M\to B\Z_{2}^{n}$, where $B\Z_{2}^{n}=(\R P^{\infty})^{n}$.

%%%%%%%%%%%%%%%%%%%%%%%%%%%%%%%%%%%%%%%%%%%%%%%%%%%%%%%%%%%%%%%%%%%%%%%%%%%
%Section 3
%%%%%%%%%%%%%%%%%%%%%%%%%%%%%%%%%%%%%%%%%%%%%%%%%%%%%%%%%%%%%%%%%%%%%%%%%%%
\section{General facts of projective bundles}
\label{sect3}

In this section, we recall some general notations and basic facts for projective bundles (see e.g. \cite{CF, Nakaoka} for details).
We first recall the definition of the projective bundle.
Let $\xi$ be a $k$-dimensional, real vector bundle over $M$.
We will denote the total space of $\xi$ by $E(\xi)$, the projection from $E(\xi)$ onto $M$ by $\widetilde{\rho}$, and
the fibre on $x\in M$ by $F_{x}(\xi)$, i.e., $F_{x}(\xi)=\widetilde{\rho}^{-1}(x)$.
Put $\xi_{0}$ the bundle induced by $\xi$ removing the $0$-section.
Then each fibre of $\xi$ has the multiplicative action of $\R^{*}=\R\setminus\{0\}$.
Taking its orbit space, we have the fibre bundle $\rho:P(\xi)\to M$ whose fibre is the $(k-1)$-dimensional real projective space $\R P^{k-1}$.
We call $P(\xi)$ the {\it projective bundle} of $\xi$.
We often denote the fibre of $P(\xi)$ on $x\in M$ by $P_{x}(\xi)$, i.e., $P_{x}(\xi)=\rho^{-1}(x)$.

We next recall the properties of cohomology of projective bundles.
Let $\iota:\R P^{k-1}\simeq P_{x}(\xi)\to P(\xi)$ be the natural embedding.
As is well known, the induced ring homomorphism
\begin{eqnarray}
\label{surjection of proj-bdl}
H^{*}(P(\xi))\stackrel{\iota^{*}}{\longrightarrow} H^{*}(\R P^{k-1})
\end{eqnarray}
is surjective.
On the other hand, the induced ring homomorphism
\begin{eqnarray}
\label{injection of proj-bdl}
H^{*}(M)\stackrel{\rho^{*}}{\longrightarrow} H^{*}(P(\xi))
\end{eqnarray}
is injective.
Moreover, we have the kernel of $\iota^{*}$ is the ideal generated by ${\rm Im}\ \rho^{+}$,
where ${\rm Im}\ \rho^{+}=\rho^{*}(H^{+}(M))$.
We want to consider the ring structure of the cohomology $H^{*}(P(\xi))$.
In order to do this,
we define the following line bundle over $P(\xi)$ associated from $\xi$:
\begin{eqnarray}
\gamma_{\xi}=\sqcup_{x\in M}\{(L,r)\in P_{x}(\xi)\times F_{x}(\xi) \ |\ r\in L \},
\end{eqnarray}
where we regard $L\in P_{x}(\xi)$ as the line in the fibre $F_{x}(\xi)$ of $\xi$.
We call $\gamma_{\xi}$ the {\it tautological (real) line bundle} of $P(\xi)$. %and denote it by $E(\gamma_{\xi})\to P(\xi)$.
Note that we have the following diagram:
\begin{eqnarray}
\label{diagram-of-bdls}
\xymatrix{
& E(\gamma_{\xi}) \ar[dr]_{\R} \ar[r] & E(\rho^{*}\xi) \ar[d]_{\R^{k}} \ar[r] & E(\xi) \ar[d]^{\widetilde{\rho}}_{\R^{k}} \\
& \R P^{k-1} \ar[r]^{\iota} & P(\xi) \ar[r]^{\rho} & M
}
\end{eqnarray}
%\begin{eqnarray}
%\label{diagram-of-bdls}
%\begin{array}{rcccc}
%E(\gamma_{\xi}) & \hookrightarrow                   & E(\rho^{*}\xi) & \longrightarrow                  & E(\xi)                      \\
%                & \searrow                          & \downarrow     &                                  & \downarrow \widetilde{\rho} \\
%\R P^{k-1}      & \stackrel{\iota}{\hookrightarrow} & P(\xi)         & \stackrel{\rho}{\longrightarrow} & M
%\end{array}
%\end{eqnarray}
where $\rho^{*}\xi$ is the pull-back of $\xi$ by $\rho$.
Let $w_{i}(\xi)\in H^{i}(M)$ be the $i^{\rm th}$ Stiefel-Whitney class of the $k$-dimensional vector bundle $\xi$ for $i=1,\ \ldots,\ k$,
and $w_{1}(\gamma_{\xi})\in H^{1}(P(\xi))$ be the $1^{\rm st}$ Stiefel-Whitney class of $\gamma_{\xi}$.
Then $\iota^{*}(w_{1}(\gamma_{\xi}))$ is the ring generator of $H^{*}(\R P^{k-1})$.
Because
\begin{eqnarray}
\label{formula-projective-space}
H^{*}(\R P^{k-1})\simeq \Z/2\Z[a]/\langle a^{k} \rangle
\end{eqnarray}
for $\deg a=1$, we have $\iota^{*}(w_{1}(\gamma_{\xi})^{k})=0$ in $H^{*}(\R P^{k-1})$.
However,
$w_{1}(\gamma_{\xi})^{k}$ might not be zero in $H^{k}(P(\xi))$.
The following formula, called the {\it Borel-Hirzebruch formula}, tells us the explicit formula of this element (see \cite{Borel-H} or \cite[(23.3)]{CF}):
\begin{eqnarray}
\label{BH}
w_{1}(\gamma_{\xi})^{k}=\sum_{i=1}^{k}\rho^{*}(w_{i}(\xi))w_{1}(\gamma_{\xi})^{k-i}.
\end{eqnarray}
Therefore, together with \eqref{injection of proj-bdl}, $H^{*}(P(\xi))$ is isomorphic to
\begin{eqnarray}
\label{algebra structure}
H^{*}(M)[x]/\langle \sum_{i=0}^{k}\rho^{*}(w_{i}(\xi)) x^{k-i} \rangle
\end{eqnarray}
as the $H^{*}(M)$-algebra, where $x=w_{1}(\gamma_{\xi})$ and $\rho^{*}(w_{i}(\xi))$ is regarded as the element in $H^{*}(M)$ (because of the injectivity of $\rho^{*}$).
Moreover, by using the Borel-Hirzebruch formula, we have the following proposition.
%%%%%%%%%%%%%%%%%%%%%%%%%%%%%%%%%%%%%%%%%%%%%%%%%%%%%%%%%%%%%%%%%%%%%%%%%%%
%Proposition 3.1
%%%%%%%%%%%%%%%%%%%%%%%%%%%%%%%%%%%%%%%%%%%%%%%%%%%%%%%%%%%%%%%%%%%%%%%%%%%
\begin{proposition}
\label{cohomology}
Let $M$ be a closed manifold, and $\xi$ a $k$-dimensional real vector bundle, where $k>1$.
Then the following two statements are equivalent:
\begin{enumerate}
\item $H^{*}(P(\xi))\simeq H^{*}(M\times \R P^{k-1})$;
\item $w(\xi)=(1+X)^{k}$ for some $kX\equiv_{2} w_{1}(\xi)\in H^{1}(M)$, where $k\equiv_{2} 0,\ 1$.
\end{enumerate}
\end{proposition}
\begin{proof}
Suppose that
the cohomology ring of $M$ satisfies that
\begin{eqnarray}
\label{formula-M}
H^{*}(M)=\Z/2\Z[\eta_{1}, \cdots, \eta_{m}]/\langle f_{j}\ |\ j=1, \ldots, l \rangle
\end{eqnarray}
for some polynomial $f_{j}=f_{j}(\eta_{1}, \cdots, \eta_{m})$ and generators $\eta_{1},\ \ldots,\ \eta_{m}$.
Because $\rho^{*}:H^{*}(M)\to H^{*}(P(\xi))$ is injective, we may regard $\eta_{1},\ \ldots,\ \eta_{m}$ as the generators in $H^{*}(P(\xi))$.
Moreover, since $\iota^{*}(w_{1}(\gamma_{\xi}))$ is the generator of $H^{*}(\R P^{k-1})$,
we may denote the cohomology ring of $H^{*}(P(\xi))$ as follows:
\begin{eqnarray}
\label{formula}
& &H^{*}(P(\xi); \Z/2\Z) \\
&\simeq & \Z/2\Z[\eta_{1}, \ldots, \eta_{m},\ w_{1}(\gamma_{\xi})]/\langle f_{j},\
w_{1}(\gamma_{\xi})^{k}-\sum_{i=1}^{k}\rho^{*}(w_{i}(\xi))w_{1}(\gamma_{\xi})^{k-i}\ |\ j=1,\ldots, l \rangle \nonumber
\end{eqnarray}
by using the Borel-Hirzebruch formula \eqref{BH}.

Assume that the statement (1) holds, that is, $H^{*}(P(\xi))\simeq H^{*}(M\times \R P^{k-1})$.
Then, we may put
\begin{eqnarray}
\label{assumption(1)}
H^{*}(P(\xi))
\simeq \Z/2\Z[\eta_{1}, \ldots, \eta_{m},\ A]/\langle f_{j},\ A^{k}\ |\ j=1,\ldots,l \rangle,
\end{eqnarray}
for some $A\in H^{1}(P(\xi))$.
%Let $\{\eta_{i_{h}}\ |\ h=1,\ldots,a\}$ be the set of degree one generators in $\{\eta_{1},\ldots,\eta_{m}\}$.
Comparing (\ref{formula}) and (\ref{assumption(1)}), we may write
\begin{eqnarray}
\label{formula-A}
A &=& w_{1}(\gamma_{\xi})+ \epsilon_{1}\eta_{1}+\cdots +\epsilon_{m}\eta_{m} \\ %\epsilon_{i_{1}}\eta_{i_{1}}+\cdots +\epsilon_{i_{a}}\eta_{i_{a}} \\
  &=& w_{1}(\gamma_{\xi})+X \nonumber
\end{eqnarray}
for some $\epsilon_{i}\in \Z/2\Z$ ($i=1,\ \ldots,\ m$).
Therefore, we have
\begin{eqnarray}
\label{equation-A}
A^{k} &\stackrel{(\ref{formula-A})}{=}& (w_{1}(\gamma_{\xi})+X)^{k}
\equiv_{2} w_{1}(\gamma_{\xi})^{k}+\sum_{i=1}^{k}
\left(
\begin{array}{c}
k \\
i
\end{array}
\right)
X^{i}w_{1}(\gamma_{\xi})^{k-i}  \\ \nonumber
&\stackrel{(\ref{BH})}{\equiv_{2}}&
\sum_{i=1}^{k}\rho^{*}(w_{i}(\xi))w_{1}(\gamma_{\xi})^{k-i}+\sum_{i=1}^{k}
\left(
\begin{array}{c}
k \\
i
\end{array}
\right)
X^{i}w_{1}(\gamma_{\xi})^{k-i} \\ \nonumber
&\stackrel{(\ref{assumption(1)})}{\equiv_{2}}& 0.
\end{eqnarray}
Due to the $H^{*}(M)$-algebraic structure of $H^{*}(P(\xi))$ in \eqref{algebra structure}, we have that
$w_{1}(\gamma_{\xi})^{0},\ \ldots,\ w_{1}(\gamma_{\xi})^{k-1}$ are the $H^{*}(M)$-module generators of $H^{*}(P(\xi))$.
Therefore,
the equation \eqref{equation-A} implies that
\begin{eqnarray*}
\rho^{*}(w_{i}(\xi))\equiv_{2}\left(
\begin{array}{c}
k \\
i
\end{array}
\right)
X^{i}.
\end{eqnarray*}
Hence, because $\rho^{*}$ is injective, we may denote
\begin{eqnarray*}
w(\xi)=(1+X)^{k},
\end{eqnarray*}
and
\begin{eqnarray*}
w_{1}(\xi)\equiv_{2}kX
\end{eqnarray*}
 where $k\equiv_{2}0$ or $1$, and $w(\xi)$ is the total Stiefel-Whitney class of $\xi$.
This establishes the statement (2).

Assume that the statement (2) holds, that is, $w(\xi)=(1+X)^{k}$.
By using (\ref{BH}) and the injectivity of $\rho^{*}$, one can easily show that $(w_{1}(\gamma_{\xi})+X)^{k}=0$.
Using (\ref{formula-projective-space}) and (\ref{formula-M}), we may put
\begin{eqnarray}
\label{product-formula}
H^{*}(M\times \R P^{k-1})=\Z/2\Z[\eta_{1}, \ldots, \eta_{m},\ a]/\langle f_{j},\ a^{k}\ |\ j=1,\ldots,l \rangle,
\end{eqnarray}
for some $a\in H^{1}(M\times \R P^{k-1})$.
%Because $w(\xi)=(1+X)^{k}$, the element $X$ may be described as follows:
%\begin{eqnarray*}
%X=\epsilon_{1}\eta_{1}+\cdots +\epsilon_{m}\eta_{m}
%\end{eqnarray*}
%for some $\epsilon_{i}\in \Z/2\Z$ ($i=1,\ \ldots,\ m$). %and $\eta_{i_{h}}$ with $\deg \eta_{i_{h}}=1$.
Therefore, using (\ref{formula}) and the above (\ref{product-formula}), there is the following isomorphism from $H^{*}(M\times \R P^{k-1})$ to $H^{*}(P(\xi))$:
\begin{eqnarray*}
& &\varphi: \eta_{i}\mapsto \eta_{i}\quad (i=1,\ \ldots,\ m); \\
& &\varphi: a \mapsto w_{1}(\gamma_{\xi})+X.
\end{eqnarray*}
This establishes the statement (1).
\end{proof}

\section{Projective bundles over small covers}
\label{sect4}

In this section, we introduce some notations and basic facts for projective bundles over small covers.
We first recall the definition of a $G$-equivariant vector bundle over $G$-space $M$ (also see the notations in Section \ref{sect3}).
A {\it $G$-equivariant vector bundle} is a vector bundle $\xi$ over $G$-space $M$
together with a lift of the $G$-action to $E(\xi)$ by fibrewise linear transformations,
i.e., $E(\xi)$ is also a $G$-space, the projection $E(\xi)\to M$ is $G$-equivariant and
the induced fibre isomorphism between $F_{x}(\xi)$ and $F_{gx}(\xi)$ is
linear, for all $x\in M$ and $g\in G$.

Before we state the first main result, we give two examples of small covers which is constructed by projectivization of vector bundles:
\begin{example}[generalized real Bott manifold]
A {\it generalized real Bott manifold} of height $m$ is an  
iterated real projective fibration defined as a sequence of real projective fibrations
\begin{eqnarray*}
\xymatrix{
%& & E_{m-1} \ar[d]^{\C^{n_{m}+1}} & \cdots & E_{1}\ar[d]^{\C^{n_{2}+1}} & E_{0}\ar[d]^{\C^{n_{1}+1}} \\
& \R B_{m} \ar[r]^{\pi_{m} \quad } & \R B_{m-1} \ar[r]^{\pi_{{m-1}}} & \cdots  \ar[r]^{\pi_{2} } & \R B_{1} \ar[r]^{\pi_{1} \quad \quad \quad } & \R B_{0}=\{\text{a point}\}
}
\end{eqnarray*}
where $\R B_{i}=P(\gamma_{i_{1}}\oplus\cdots\gamma_{i_{l}})$ is the projectivization of a 
Whitney sum of line bundles over $\R B_{i-1}$.
Note that $\R B_{1}$ is just the real projective space.
If the dimension of each fibre is exactly $1$, then this is called a {\it real Bott manifold}.
See \cite{M} for details.
\end{example}

\begin{example}[Stong manifold]
Let $\pi_{i}:B=\R P^{n_{1}}\times\cdots\times \R P^{n_{l}}\to \R P^{n_{i}}$ be the natural projection, for $i=1,\ldots,l$.
We define the line bundle $\gamma_{i}$ over $B$ by the pull-back of the tautological line bundle over $\R P^{n_{i}}$ along $\pi_{i}$.
Then, a {\it Stong manifold} $S$ is defined by the following projectivization over $B$:
\[
S=P(\gamma_{1}\oplus\cdots\oplus\gamma_{l})\to B.
\]
It is easy to check that this is a generalized real Bott manifold.
\end{example}

%%%%%%%%%%%%%%%%%%%%%%%%%%%%%%%%%%%%%%%%%%%%%%%%%%%%%%%%%%%%%%%%%%%%%%%%%%%
%Section 4.1
%%%%%%%%%%%%%%%%%%%%%%%%%%%%%%%%%%%%%%%%%%%%%%%%%%%%%%%%%%%%%%%%%%%%%%%%%%%
\subsection{Necessary and sufficient conditions of when $P(\xi)$ is a small cover}
\label{sect4.1}

From this section to Section \ref{sect6}, we assume $M$ is an $n$-dimensional small cover, and
$\xi$ is a $k$-dimensional, $\Z_{2}^{n}$-equivariant vector bundle over $M$.
The following proposition gives a criterion for the projective bundle $P(\xi)$ to be a small cover:
%%%%%%%%%%%%%%%%%%%%%%%%%%%%%%%%%%%%%%%%%%%%%%%%%%%%%%%%%%%%%%%%%%%%%%%%%%%
%Proposition 4.1
%%%%%%%%%%%%%%%%%%%%%%%%%%%%%%%%%%%%%%%%%%%%%%%%%%%%%%%%%%%%%%%%%%%%%%%%%%%
\begin{theorem}
\label{basic1}
The projective bundle $P(\xi)$ of $\xi$ is a small cover if and only if the equivariant vector bundle $\xi$ decomposes into the Whitney sum of line bundles, i.e.,
$\xi\equiv\gamma_{1}\oplus\cdots\oplus\gamma_{k-1}\oplus\gamma_{k}$.
\end{theorem}
\begin{proof}
Assume that $P(\xi)$ is a small cover.
By definition, $P(\xi)$ has a locally standard $(\Z_{2}^{n}\times \Z_{2}^{k-1})$-action.
Because $\xi$ is a $k$-dimensional, $\Z_{2}^{n}$-equivariant vector bundle over $M$,
the projection $\rho:P(\xi)\to M$ is $\Z_{2}^{n}$-equivariant.
In particular, the $\Z_{2}^{k-1}$-action is trivially acts on $M$.
Therefore, each fibre $P_{x}(\xi)$ over $x\in M$ has an effective $\Z_{2}^{k-1}$-action.
This implies that there is the $\Z_{2}^{k}$-action on $F_{x}(\xi)$ such that
$(F_{x}(\xi)\setminus \{0\})/\R^{*}$ is $\Z_{2}^{k-1}$-equivariantly homeomorphic to $P_{x}(\xi)$,
where $F_{x}(\xi)\cong \R^{k}$ is the fibre of $E(\xi)$ over $x\in M$.
Hence, the total space $E(\xi)$ of $\xi$ has a $(\Z_{2}^{n}\times \Z_{2}^{k})$-action and
the restricted $\Z_{2}^{n}$-action is induced from the lift of the $\Z_{2}^{n}$-action on $M$.
Let $\{U_{i}\}_{i\in \mathcal{I}}$ be a $\Z_{2}^{n}$-equivariant open covering of $M$.
Then, by using the local triviality condition of the vector bundle,
we may denote $\xi$ as the gluing of $U_{i}\times \R^{k}$ for $i\in \mathcal{I}$, say $\amalg_{i\in \mathcal{I}} (U_{i}\times \R^{k})/\sim$.
Here, the symbol $\sim$ represents the identification $(u,x)\sim (u,y)$ for
$u\in U_{i}\cap U_{j}$ by $x=A(u)y\in\R^{k}$ for some transition function $A(u)\in GL(k;\R)$. %which satisfies \underline{the cocycle condition}.
Here, because $M$ is a small cover (in particular smooth closed manifold), we may reduce the structure group into the orthogonal group $O(k)$ and we can take $A(u)\in O(k)$.
Therefore, if the $\Z_{2}^{k}$-action on the $\R^{k}$-factor in $U_{i}\times \R^{k}$ extends to
the global action on $\amalg (U_{i}\times \R^{k})/\sim$,
then the transition function $A(u)\in O(k)$ must commute with $\Z_{2}^{k}$ for all $u\in U_{i}\cap U_{j}$.
Note that we may regard $\Z_{2}^{k}$ as the diagonal subgroup in $O(k)$ up to conjugation.
Because the centralizer of $\Z_{2}^{k}$ (the diagonal subgroup) in $O(k)$ is $\Z_{2}^{k}$ (the diagonal subgroup) itself,
we have $A(u)\in \Z_{2}^{k}\subset O(k)$ for all $u\in U_{i}\cap U_{j}$.
%Because the $(\Z_{2}^{n}\times \Z_{2}^{k-1})$-action on $P(\xi)$ is locally standard, the $(\Z_{2}^{n}\times \Z_{2}^{k})$-action on $E(\xi)$ is also locally standard.
%Therefore, we may regard $A(u)$ as a diagonal element in $O(k)$ for all $u\in M$.
This implies that the structure group of $\xi$ is $\Z_{2}^{k}$. 
%, i.e., the subgroup which generated by all diagonal elements in $O(k)$.
This is nothing but
$\xi\equiv\gamma_{1}\oplus\cdots\oplus\gamma_{k-1}\oplus\gamma_{k}$.

Conversely, if $\xi\equiv\gamma_{1}\oplus\cdots\oplus\gamma_{k-1}\oplus\gamma_{k}$, then we can easily check that
this vector bundle has the $\Z_{2}^{k}$-action along fibre and $P(\xi)$ has the induced locally standard $(\Z_{2}^{n}\times \Z_{2}^{k-1})$-action.
\end{proof}

As is well known, $P(\xi\otimes\gamma)\cong P(\xi)$ (homeomorphic) for all line bundle $\gamma$ (e.g. see \cite{MS}).
Hence, by using the above Proposition \ref{basic1}, the following corollary holds.
%%%%%%%%%%%%%%%%%%%%%%%%%%%%%%%%%%%%%%%%%%%%%%%%%%%%%%%%%%%%%%%%%%%%%%%%%%%
%Corollary 4.2
%%%%%%%%%%%%%%%%%%%%%%%%%%%%%%%%%%%%%%%%%%%%%%%%%%%%%%%%%%%%%%%%%%%%%%%%%%%
\begin{corollary}
\label{cor-basic}
Let $M$ be a small cover, and $\xi$ be a Whitney sum of $k$ line bundles over $M$.
Then the small cover $P(\xi)$ is homeomorphic to
\begin{eqnarray*}
P(\gamma_{1}\oplus\cdots\gamma_{k-1}\oplus\epsilon),
\end{eqnarray*}
where $\epsilon$ is the trivial line bundle over $M$.
\end{corollary}
\begin{proof}
Assume $\xi\equiv \gamma_{1}'\oplus\cdots\oplus\gamma_{k-1}'\oplus\gamma_{k}'$.
Then we have that
\begin{eqnarray*}
P(\gamma_{1}'\oplus\cdots\oplus\gamma_{k-1}'\oplus\gamma_{k}')
\cong P((\gamma_{1}'\otimes\gamma_{k}')\oplus\cdots\oplus(\gamma_{k-1}'\otimes\gamma_{k}')\oplus\epsilon),
\end{eqnarray*}
because $\gamma_{k}'\otimes\gamma_{k}'\equiv\epsilon$.
This establishes the statement.
\end{proof}

In this paper, the projective bundle in Corollary \ref{cor-basic} (also see Section \ref{sect1}) is said to be the {\it projective bundle over small cover}.

%%%%%%%%%%%%%%%%%%%%%%%%%%%%%%%%%%%%%%%%%%%%%%%%%%%%%%%%%%%%%%%%%%%%%%%%%%%
%Section 4.2
%%%%%%%%%%%%%%%%%%%%%%%%%%%%%%%%%%%%%%%%%%%%%%%%%%%%%%%%%%%%%%%%%%%%%%%%%%%
\subsection{Structures of projective bundles over small covers}
\label{sect4.2}

In this subsection, we show the quotient construction of the projective bundles of small covers.
First, we recall the moment-angle manifold of small covers (see \cite{BP, DJ}).
Let $P$ be a simple, convex polytope and
$\mathcal{F}$ the set of its facets $\{F_{1},\ \cdots,\ F_{m}\}$.
We denote by $\mathcal{Z}_{P}$ the manifolds
\begin{eqnarray*}
\mathcal{Z}_{P}=\Z_{2}^{m}\times P/\sim,
\end{eqnarray*}
where $(t,\ p)\sim (t',\ p)$ is defined by $t^{-1}t'\in \prod_{p\in F_{i}}\Z_{2}(i)$
($\Z_{2}(i)\subset \Z_{2}^{m}$ is the rank $1$ subgroup generated by the $i$-th factor),
and we call it a {\it moment-angle manifold} of $P$.
We note that if $P=M^{n}/\Z_{2}^{n}$ then there is the subgroup $K\subset \Z_{2}^{m}$ such that $K\simeq \Z_{2}^{m-n}$ and $K$ acts freely on $\mathcal{Z}_{P}$.
Therefore, we can denote the small cover $M=\mathcal{Z}_{P}/\Z_{2}^{l}$ by the free $\Z_{2}^{l}$-action on $\mathcal{Z}_{P}$ for $l=m-n$.

Since $[M;\ B\Z_{2}]=H^{1}(M;\ \Z_{2})\simeq \Z_{2}^{l}$ (see \cite{DJ, S}),
we see that all line bundles $\gamma$ can be written as follows:
\begin{eqnarray*}
\gamma\equiv \mathcal{Z}_{P}\times_{\Z_{2}^{l}}\R_{\alpha},
\end{eqnarray*}
where $\Z_{2}^{l}$ acts on $\R_{\alpha}=\R$ by some representation $\alpha:\Z_{2}^{l}\to \Z_{2}$.
Moreover, its total Stiefel-Whitney class is
$w(\mathcal{Z}_{P}\times_{\Z_{2}^{l}}\R)=1+\delta_{1}x_{1}+\cdots +\delta_{l}x_{l}$,
where $(\delta_{1}, \cdots, \delta_{l})\in \{0,\ 1\}^{l}$
is induced by a representation $\Z_{2}^{l}\to \Z_{2}$, i.e.,
\begin{eqnarray*}
(\epsilon_{1},\ \cdots,\ \epsilon_{l})\mapsto \epsilon_{1}^{\delta_{1}}\cdots \epsilon_{l}^{\delta_{l}},
\end{eqnarray*}
for $\epsilon_{i}\in \Z_{2}$,
and $x_{1},\ \ldots,\ x_{l}$ are the degree $1$ generators of $H^{*}(M)$ introduced in Section \ref{sect2.3}.
Therefore, by using Corollary \ref{cor-basic}, all projective bundles of small covers are as follows:
\begin{eqnarray}
\label{proj}
P(\xi)=\mathcal{Z}_{P}\times_{\Z_{2}^{l}}(\R^{k}\setminus \{0\})/\R^{*}=\mathcal{Z}_{P}\times_{\Z_{2}^{l}}\R P^{k-1},
\end{eqnarray}
where
\begin{eqnarray*}
\xi=\mathcal{Z}_{P}\times_{\Z_{2}^{l}}\R^{k}
\end{eqnarray*}
with the $\Z_{2}^{l}$-representation space $\R^{k}=\R_{\alpha_{1}}\oplus\cdots\oplus\R_{\alpha_{k}}$ such that
\begin{eqnarray*}
\alpha_{i}:\Z_{2}^{l}\to \Z_{2}
\end{eqnarray*}
where $i=1,\ \cdots,\ k$ and $\alpha_{k}$ is the trivial representation.
Then,
we may denote the projective bundle of small cover by
\begin{eqnarray*}
\mathcal{Z}_{P}\times_{\Z_{2}^{l}}\R P^{k-1}=P(\gamma_{1}\oplus\cdots\gamma_{k-1}\oplus\epsilon),
\end{eqnarray*}
where $\gamma_{i}=\mathcal{Z}_{P}\times_{\Z_{2}^{l}}\R_{\alpha_{i}}$ ($i=1,\ \cdots,\ k-1$) satisfies
$w(\gamma_{i})=1+\delta_{1i}x_{1}+\cdots +\delta_{li}x_{l}$
 for $(\delta_{1i},\ \cdots,\ \delta_{li})\in (\Z/2\Z)^{l}$, which is induced by the representation
$\alpha_{i}:\Z_{2}^{l}\to \Z_{2}$.
This is also denoted by the following form:
\[
\mathcal{Z}_{P}\times_{\Z_{2}^{l}}P(\R_{\alpha_{1}}\oplus\cdots\oplus\R_{\alpha_{k}}).
\]

Let $(I_{n} | \Lambda)\in M(m,n;\Z/2\Z)$ be the characteristic matrix of $M$.
Using the above constriction of projective bundles and computing their characteristic functions,
we have the following proposition.
%%%%%%%%%%%%%%%%%%%%%%%%%%%%%%%%%%%%%%%%%%%%%%%%%%%%%%%%%%%%%%%%%%%%%%%%%%%
%Proposition 4.3
%%%%%%%%%%%%%%%%%%%%%%%%%%%%%%%%%%%%%%%%%%%%%%%%%%%%%%%%%%%%%%%%%%%%%%%%%%%
\begin{proposition}
\label{ch-fct}
Let $P(\gamma_{1}\oplus\cdots\oplus\gamma_{k-1}\oplus\epsilon)$ be the projective bundle over $M$.
Then its orbit polytope is $P^{n}\times \Delta^{k-1}$, and
its characteristic matrix is as follows:
\begin{eqnarray}
\label{ch-matrix}
\left(
\begin{array}{cccc}
I_{n} & O       & \Lambda        & \textbf{0} \\
O     & I_{k-1} & \Lambda_{\xi}  & \textbf{1}
\end{array}
\right),
\end{eqnarray}
where $P^{n}=M/\Z_{2}^{n}$ and
\begin{eqnarray*}
\Lambda_{\xi}=
\left(
\begin{array}{ccc}
\delta_{11}    & \cdots & \delta_{l1} \\
\vdots         & \ddots & \vdots \\
\delta_{1,k-1} & \cdots & \delta_{l,k-1}
\end{array}
\right).
\end{eqnarray*}
\end{proposition}

%%%%%%%%%%%%%%%%%%%%%%%%%%%%%%%%%%%%%%%%%%%%%%%%%%%%%%%%%%%%%%%%%%%%%%%%%%%
%Corollary 4.4
%%%%%%%%%%%%%%%%%%%%%%%%%%%%%%%%%%%%%%%%%%%%%%%%%%%%%%%%%%%%%%%%%%%%%%%%%%%
Therefore, we have the following corollary by using Proposition \ref{cohomology}.
\begin{corollary}
Let $M$ be an small cover, and $\xi$ a Whitney sum of $k$ line bundles over $M$.
Then the following two statements are equivalent:
\begin{enumerate}
\item $H^{*}(P(\xi);\ \Z_{2})\simeq H^{*}(M\times \R P^{k-1};\ \Z_{2})$;
\item $w(\xi)=w(\gamma_{1}\oplus\cdots\oplus\gamma_{k-1}\oplus\epsilon)=\prod_{j=1}^{k-1}(1+\sum_{i=1}^{l}\delta_{ij}x_{i})=(1+X)^{k}$.
\end{enumerate}
\end{corollary}

%%%%%%%%%%%%%%%%%%%%%%%%%%%%%%%%%%%%%%%%%%%%%%%%%%%%%%%%%%%%%%%%%%%%%%%%%%%
%Section 4.3
%%%%%%%%%%%%%%%%%%%%%%%%%%%%%%%%%%%%%%%%%%%%%%%%%%%%%%%%%%%%%%%%%%%%%%%%%%%
\subsection{New characteristic function of projective bundles over small covers}
\label{sect4.3}
In order to show the construction theorem of projective bundles over $2$-dimensional small covers,
we introduce a new characteristic function (matrix).
Let $(I_{n}\ |\ \Lambda)$ be the characteristic matrix of $M^{n}$, where $\Lambda\in M(n, l;\Z/2\Z)$ for $l=m-n$ ($m$ is the number of facets of $P^{n}=M/\Z_{2}^{n}$).
By using Proposition \ref{ch-fct},
the characteristic matrix of $P(\gamma_{1}\oplus\cdots\oplus\gamma_{k-1}\oplus\epsilon)$ is
\begin{eqnarray}
\label{ch-matrix2}
\left(
\begin{array}{cccc}
I_{n} & O       & \Lambda        & \textbf{0} \\
O     & I_{k-1} & \Lambda_{\xi}  & \textbf{1}
\end{array}
\right),
\end{eqnarray}
where
\begin{eqnarray*}
\left(
\begin{array}{c}
\Lambda \\
\hline
\Lambda_{\xi}
\end{array}
\right)=
\left(
\begin{array}{ccc}
\lambda_{11}   & \cdots & \lambda_{l1} \\
\vdots           & \ddots & \vdots \\
\lambda_{1n}   & \cdots & \lambda_{ln} \\
\hline
\delta_{11}    & \cdots & \delta_{l1} \\
\vdots         & \ddots & \vdots \\
\delta_{1,k-1} & \cdots & \delta_{l,k-1}
\end{array}
\right)=
\left(
\begin{array}{ccc}
\textbf{a}_{1}   & \cdots & \textbf{a}_{l} \\
\hline
\textbf{b}_{1}    & \cdots & \textbf{b}_{l}
\end{array}
\right),
\end{eqnarray*}
where $\textbf{a}_{i}\in \{0,\ 1\}^{n}$ and $\textbf{b}_{i}\in \{0,\ 1\}^{k-1}$ for $i=1,\ \cdots,\ l$.
Therefore, in order to characterize the projective bundles over $M^{n}$, it is sufficient to attach the following $((n+k-1)\times m)$-matrix
\begin{eqnarray}
\label{proj-ch-matrix}
\left(
\begin{array}{cc}
I_{n} & \Lambda \\
O & \Lambda_{\xi}
\end{array}
\right)
\end{eqnarray}
on the facets of $P$.
Namely, it is sufficient to consider the following
characteristic function:
for $P$ and its facets $\mathcal{F}$, the function
\begin{eqnarray*}
\lambda_{P}:\mathcal{F}=\{F_{1},\ \cdots,\ F_{m}\}\to \{0,\ 1\}^{n}\times \{0,\ 1\}^{k-1}
\end{eqnarray*}
satisfies $\lambda_{P}(F_{1})=e_{1}\times \textbf{0},\ \ldots,\ \lambda_{P}(F_{n})=e_{n}\times \textbf{0}$ (where $e_{i}$ is the standard basis in $(\Z/2\Z)^{n}$)
and the projection to the 1st factor $p_{a}:\{0,\ 1\}^{n}\times \{0,\ 1\}^{k-1}\to \{0,\ 1\}^{n}$
satisfy that
\begin{eqnarray}
\label{proj-smooth}
\det (p_{a}\circ \lambda_{P}(F_{i_{1}})\cdots p_{a}\circ \lambda_{P}(F_{i_{n}}))=1
\end{eqnarray}
if $F_{i_{1}}\cap\cdots\cap F_{i_{n}}\not=\emptyset$, i.e., $p_{a}\circ \lambda_{P}$ is the usual characteristic function on $P^{n}$.
We call this function $\lambda_{P}$ a {\it projective characteristic function} over $P^{n}$; or a
{\it $(k-1)$-dimensional projective characteristic function} when we emphasize the dimension of the fibre.
One can easily show that the $((n+k-1)\times m)$-matrix $(\lambda_{P}(F_{1})\cdots \lambda_{P}(F_{m}))$ is identified with the matrix \eqref{proj-ch-matrix} up to isomorphism.
We call this matrix a {\it $($(k-1)-dimensional$)$ projective characteristic matrix} over $P$.
Figure \ref{new-ch} is an illustration of projective characteristic functions over $2$-dimensional small covers.
\begin{figure}[h]
\includegraphics[width=200pt,clip]{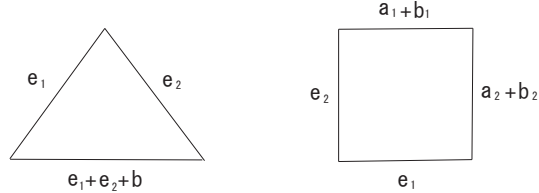}
\caption{The examples of $(k-1)$-dimensional projective characteristic functions.
Here, $e_{1}=(1,0)\times \textbf{0}$ and $e_{2}=(0,1)\times \textbf{0}$ are the generators of $(\Z/2\Z)^{2}\times \textbf{0}$.
For the left triangle, we may take an arbitrary element $b\in (0,0)\times (\Z/2\Z)^{k-1}$ and one can easily show that each of those corresponds with a projective bundle over $\R P^{2}$ whose fibre is $\R P^{k-1}$.
For the right square, $a_{1},\ a_{2}$ are elements in $(\Z/2\Z)^{2}\times \textbf{0}$ which satisfy (\ref{proj-smooth}) on each vertex, and $b_{1},\ b_{2}\in (0,0)\times (\Z/2\Z)^{k-1}$ determine the bundle structure.}
\label{new-ch}
\end{figure}

Note that in Figure \ref{new-ch},
if we put $b=0$ and $b_{1}=b_{2}=0\in \Z_{2}^{k-1}$, then this gives ordinary characteristic functions on the triangle and the square.
Therefore, we can regard such a forgetful map of the $(\Z/2\Z)^{k-1}$ part, $f:(P^{n},\lambda_{P})\to (P^{n},p_{a}\circ \lambda_{P})$, as
the equivariant projection $P(\xi)\to M(P,p_{a}\circ \lambda_{P})$.

Using Proposition \ref{ch-fct} and the construction method of small covers (see Section \ref{sect2.2}), one can easily show that
the pair $(P^{n},\ \lambda_{P})$ corresponds with the projective bundle over the $n$-dimensional small cover whose orbit polytope is $P^{n}$.
More precisely, for the projective bundle $P(\xi)$ over the small cover $M(P,\lambda)$ there exists
the projective characteristic function $(P,\lambda_{P})$ such that $p_{a}\circ \lambda_{P}=\lambda$.
On the other hand, for the projective characteristic function $(P,\lambda_{P})$ there exists the
projective bundle $\mathcal{Z}_{P}\times_{\Z_{2}^{l}}\R P^{k-1}=P(\gamma_{1}\oplus\cdots\oplus\gamma_{k-1}\oplus\epsilon)$ over $\mathcal{Z}_{p}/\Z_{2}^{l}=M(P,p_{a}\circ\lambda)$
up to $(\Z_{2}^{n}\times \Z_{2}^{k-1})$-equivariant homeomorphism, where the line bundle
$\gamma_{i}$ is determined by the $i$th column vector of $\Lambda_{\xi}=(p_{b}\circ \lambda_{P}(F_{n+1})\cdots p_{b}\circ \lambda_{P}(F_{m}))$ (also see Section \ref{sect4.1}),
where $p_{b}:\{0,1\}^{n}\times \{0,1\}^{k-1}\to \{0,1\}^{k-1}$ is the projection to the $2$nd factor.

Summing up, we have the following relationship:
\begin{center}
\fbox{
\begin{tabular}{c}
Projective bundles \\
over small cover $M^{n}$ \\
whose orbit polytope is $P^{n}$
\end{tabular}
}
$
\begin{array}{c}
\longrightarrow  \\
\longleftarrow
\end{array}
$
\fbox{
\begin{tabular}{c}
Simple, convex polytope $P^{n}$ \\
with projective characteristic functions $\lambda_{P}$
\end{tabular}
}
\end{center}

%%%%%%%%%%%%%%%%%%%%%%%%%%%%%%%%%%%%%%%%%%%%%%%%%%%%%%%%%%%%%%%%%%%%%%%%%%%
%Section 5
%%%%%%%%%%%%%%%%%%%%%%%%%%%%%%%%%%%%%%%%%%%%%%%%%%%%%%%%%%%%%%%%%%%%%%%%%%%
\section{New operations and main theorem}
\label{sect5}

In this section, we state our main theorem.
Before doing so, we introduce a new operation.

%%%%%%%%%%%%%%%%%%%%%%%%%%%%%%%%%%%%%%%%%%%%%%%%%%%%%%%%%%%%%%%%%%%%%%%%%%%
%Section 5.1
%%%%%%%%%%%%%%%%%%%%%%%%%%%%%%%%%%%%%%%%%%%%%%%%%%%%%%%%%%%%%%%%%%%%%%%%%%%
\subsection{Combinatorial interpretation of the fibre sum}
\label{sect5.1}

For two polytopes with projective characteristic functions,
we can do the connected sum operation which is compatible with projective characteristic functions as indicated in Figure \ref{operation-projective-bdl}.
Then we get a new polytope with the projective characteristic function.
We call this operation a {\it projective fibre sum} and denote it by $\sharp_{\Delta^{k-1}}$.

More precisely, the operation is defined as follows.
Let $p$ and $q$ be vertices in $n$-dimensional polytopes with $(k-1)$-dimensional projective characteristic functions $(P,\lambda_{P})$ and $(P',\lambda_{P'})$, respectively.
Here, we assume that the target spaces of the maps $\lambda_{P}$ and $\lambda_{P'}$ are the same $(\Z/2\Z)^{n}\times (\Z/2\Z)^{k-1}$,
i.e., the corresponding projective bundles have the same fibre $\R P^{k-1}$.
Moreover, we assume that
$\lambda_{P}(F_{i})=\lambda_{P'}(F_{i}')$ for all facets $\{F_{1},\ldots,F_{n}\}$ around $p$ and $\{F_{1}',\ldots,F_{n}'\}$ around $q$, i.e., $\cap_{i=1}^{n}F_{i}=\{p\}$ and $\cap_{i=1}^{n}F_{i}'=\{q\}$.
Then we can do the connected sum of two polytopes $P$ and $P'$ at these vertices by gluing each pair of facets $F_{i}$ and $F_{i}'$.
%(also see \cite[Definition 3]{Izmestiev}),
Thus, we get a combinatorial object (might not be a convex polytope) with projective characteristic functions $(P\sharp_{\Delta^{k-1}} P',\ \lambda_{P\sharp_{\Delta^{k-1}} P'})$ from $(P,\lambda_{P})$ and $(P',\lambda_{P'})$ (also see Figure \ref{operation-projective-bdl}).
%Because both of $P$ and $P'$ are simple convex polytope, the combinatorial object $P\sharp_{\Delta^{k-1}} P'$
%(this is just the connected sum around the vertices)
%is also a simple convex polytope.
Note that $P_{1}\sharp_{\Delta^{k-1}} P_{2}$ is a combinatorial simple convex polytope if $n\le 3$ by using {\it the Steinitz' theorem}:
the graph $\Gamma$ is a graph of the $3$-dimensional
polytope $P$ if and only if $\Gamma$ is $3$-connected and planer
 (see \cite[Chapter 4]{Ziegler}).
By following the converse of the above definition, we may define the inverse operation $\sharp_{\Delta^{k-1}}^{-1}$.

\begin{figure}[h]
\includegraphics[width=150pt,clip]{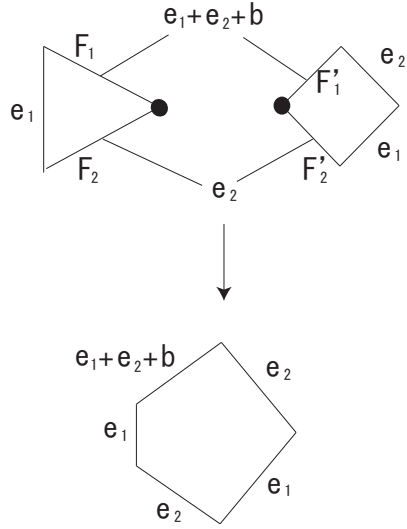}
\caption{The projective fibre sum $\sharp_{\Delta^{k-1}}$ along the same labeled vertices.}
%We can do connected sum at two vetices with same projective characteristic functions.}
\label{operation-projective-bdl}
\end{figure}

From the geometric point of view,
the inverse image of vertices of polytopes with projective characteristic functions corresponds to the projective space $\R P^{k-1}$.
Therefore,
a geometric interpretation of this operation is an equivariant gluing along the fibre $\R P^{k-1}$, i.e., fibre sum of two fibre bundles.

\begin{remark}
\label{relation with the connected sum}
If $k=1$, then the projective characteristic function is the ordinary characteristic function, i.e., the fibre dimension is $0$.
Therefore, we can regard the $0$-dimensional projective fibre sum $\sharp_{\Delta^{0}}$ as the ordinary (equivariant) connected sum $\sharp$ appeared in \cite{Izmestiev, Kuroki, LY, Nishimura}.
\end{remark}

If $P\sharp_{\Delta^{k-1}}P'$ is a convex simple polytope,
then  $(P\sharp_{\Delta^{k-1}} P',\ \lambda_{P\sharp_{\Delta^{k-1}} P'})$ defines
the $(k-1)$-dimensional projective bundle over $M\sharp M'$ (connected sum), where $M=M(P,p_{a}\circ\lambda_{P})$ and $M'=M(P',p_{a}\circ\lambda_{P'})$.
We note that if $\lambda_{P}$ and $\lambda_{P'}$ are
\begin{eqnarray*}
\left(
\begin{array}{ccccc}
I_{n} & \Lambda & \textbf{a}_{1} & \cdots & \textbf{a}_{n}  \\
O &  \Lambda_{\xi} &\textbf{b}_{1} & \cdots & \textbf{b}_{n}
\end{array}
\right), \quad
\left(
\begin{array}{ccccc}
I_{n} & \Lambda' & \textbf{a}_{1} & \cdots & \textbf{a}_{n}  \\
O & \Lambda_{\xi'} &\textbf{b}_{1} & \cdots & \textbf{b}_{n}
\end{array}
\right),
\end{eqnarray*}
respectively, then $\lambda_{P\sharp_{\Delta^{k-1}}P'}$ is
\begin{eqnarray*}
\left(
\begin{array}{cccccc}
I_{n} &  \Lambda & \textbf{a}_{1} & \cdots & \textbf{a}_{n} & \Lambda'  \\
O &  \Lambda_{\xi} &\textbf{b}_{1} & \cdots & \textbf{b}_{n} & \Lambda_{\xi'}
\end{array}
\right),
\end{eqnarray*}
where the same $n$ column vectors above correspond to the projective characteristic functions on $\{F_{1},\ldots,F_{n}\}$ and $\{F'_{1},\ldots,F'_{n}\}$.

%%%%%%%%%%%%%%%%%%%%%%%%%%%%%%%%%%%%%%%%%%%%%%%%%%%%%%%%%%%%%%%%%%%%%%%%%%%
%Section 5.2
%%%%%%%%%%%%%%%%%%%%%%%%%%%%%%%%%%%%%%%%%%%%%%%%%%%%%%%%%%%%%%%%%%%%%%%%%%%
\subsection{Construction theorem of projective bundles over $2$-dimensional small covers}
\label{sect5.2}

In this subsection, we prove one of the main results of this paper.
Put $P(\kappa_{i})$ and $P(\zeta_{j})$ projective bundles over $\R P^{2}$ and $T^{2}$, respectively
(also see Proposition \ref{2-proj} and \ref{2-torus}).
Here, $\kappa_{i}$ and $\zeta_{j}$ are products of $k$ line bundles, i.e., $P(\kappa_{i})$ and $P(\zeta_{j})$ have the same fibre $\R P^{k-1}$
(for $i,\ j=1,2,\ldots$).
Then, we have the following theorem.

%%%%%%%%%%%%%%%%%%%%%%%%%%%%%%%%%%%%%%%%%%%%%%%%%%%%%%%%%%%%%%%%%%%%%%%%%%%
%Theorem 5.2
%%%%%%%%%%%%%%%%%%%%%%%%%%%%%%%%%%%%%%%%%%%%%%%%%%%%%%%%%%%%%%%%%%%%%%%%%%%
\begin{theorem}
\label{construction}
Let $P(\xi)$ be a projective bundle over $2$-dimensional small cover $M^{2}$.
Then, $P(\xi)$ is weak equivariantly homeomorphic to
\begin{eqnarray*}
P(\kappa_{1})\sharp_{\Delta^{k-1}}\cdots \sharp_{\Delta^{k-1}}P(\kappa_{l_{1}})\sharp_{\Delta^{k-1}} P(\zeta_{1})\sharp_{\Delta^{k-1}}\cdots \sharp_{\Delta^{k-1}}P(\zeta_{l_{2}}),
\end{eqnarray*}
for some vector bundles $\kappa_{1},\ \ldots,\ \kappa_{l_{1}}$ and $\zeta_{1},\ \ldots,\ \zeta_{l_{2}}$,
where two $G$-manifolds $X$ and $Y$ are weak equivariantly homeomorphic
if they are equivariantly homeomorphic up to automorphism of $G$.
\end{theorem}
\begin{proof}
Let $P$ be the orbit polytope of $M$.
Because $\dim M=2$, we may assume that $P$ is an $m$-gon for some $m\ge 3$,
where $m$ is the number of facets in $P$ and we may put them $\{F_{1},\ \ldots,\ F_{m}\}$.
Moreover, we assume $F_{i}\cap F_{i+1}\not=\emptyset$ and $F_{1}\cap F_{m}\not=\emptyset$.

We first claim that, in $m$-gon for $m\ge 5$, there are two separated facets $F$ and $F'$ whose projective characteristic functions satisfy (\ref{proj-smooth}).
Assume $m\ge 5$.
Put the projective characteristic function of $F_{i}$ and $F_{j}$ (where $F_{i}\cap F_{j}=\emptyset$) as follows:
\begin{eqnarray*}
\left(
\begin{array}{c}
a_{i} \\
b_{i}
\end{array}
\right)\quad {\rm and}\quad
\left(
\begin{array}{c}
a_{j} \\
b_{j}
\end{array}
\right),
\end{eqnarray*}
respectively, where $a_{i},\ a_{j}\in \{0,\ 1\}^{2}=(\Z/2\Z)^{2}$ and $b_{i},\ b_{j}\in \{0,\ 1\}^{k-1}$.
Note that $\det (a_{i},\ a_{j})=1$ if and only if $a_{i}\not=a_{j}$ because they are elements in $(\Z/2\Z)^{2}$.
If $\det (a_{i},\ a_{j})=1$, we can take $F_{i}$ and $F_{j}$ as $F$ and $F'$ we want.
Assume $\det (a_{i},\ a_{j})=0$, i.e, $a_{i}=a_{j}$.
Since $m\ge 5$, we may assume that
the facet $F_{j+1}$ which is next to $F_{j}$, i.e, $F_{j+1}\cap F_{j}\not=\emptyset$, satisfies that $F_{j+1}\cap F_{i}=\emptyset$.
Therefore, by $a_{i}=a_{j}$, we have $\det (a_{j+1},a_{j})=\det (a_{j+1},a_{i})=1$.
Thus, we can take $F_{i}$ and $F_{j+1}$ as $F$ and $F'$ we want.
This establishes the claim.

For such facets $F$ and $F'$, we can do $\sharp_{\Delta^{k-1}}^{-1}$, because there are two $m_{1}$-gon $P_{1}$ and $m_{2}$-gon $P_{2}$ (where $m=m_{1}+m_{2}-2$)
with vertices generated by two facets which have the same projective characteristic functions of $F$ and $F'$ (see Figure\ref{explain1}).
\begin{figure}[h]
\includegraphics[width=100pt,clip]{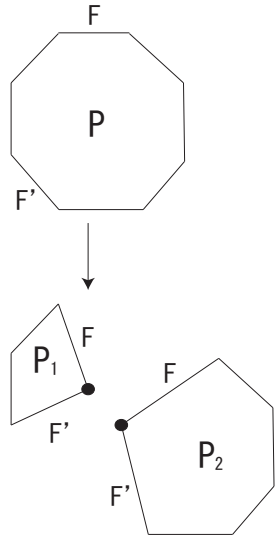}
\caption{We can do $\sharp_{\Delta^{k-1}}^{-1}$ always for $m$-gon $P$ ($m\ge 5$).
This figure illustrates the $8$-gon $P$ decomposes into the $4$-gon $P_{1}$ and the $6$-gon $P_{2}$.
Here, each $F$ (resp.\ $F'$) has the same projective characteristic function, and every corresponding facets also have the same projective characteristic functions.}
\label{explain1}
\end{figure}
This implies that $(P,\lambda_{P})$ can be constructed from $(P_{1},\lambda_{P_{1}})$ and $(P_{2},\lambda_{P_{2}})$ by using $\sharp_{\Delta^{k-1}}$,
where $P$ is an $m$-gon ($m\ge 5$), $P_{1}$ is an $m_{1}$-gon and $P_{2}$ is an $m_{2}$-gon.
Note that $m_{1}$ and $m_{2}$ are strictly less than $m$.
Iterating this argument, finally we have the finite number of $3$-gons and $4$-gons (see Figure \ref{explain2}).
\begin{figure}[h]
\includegraphics[width=120pt,clip]{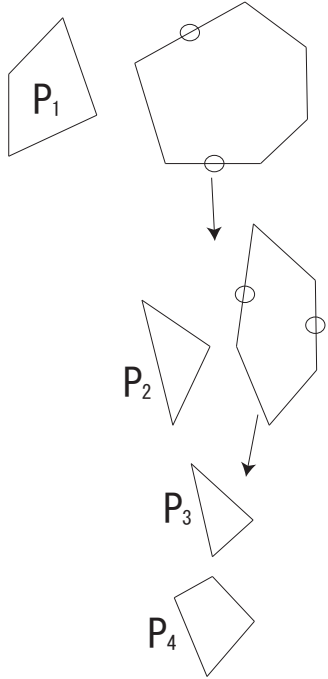}
\caption{Iterating $\sharp_{\Delta^{k-1}}^{-1}$, finally, we have finite $3$-gons and $4$-gons; $P_{1},\ \ldots,\ P_{4}$ in this case.}
\label{explain2}
\end{figure}

It is easy to see that we can not do $\sharp_{\Delta^{k-1}}$ for $3$-gons any more.
However, there are two $4$-gons; one can not do $\sharp_{\Delta^{k-1}}$ (such as the left in Figure \ref{explain3}),
and another can do $\sharp_{\Delta^{k-1}}$ (such as the right in Figure \ref{explain3}).
\begin{figure}[h]
\includegraphics[width=200pt,clip]{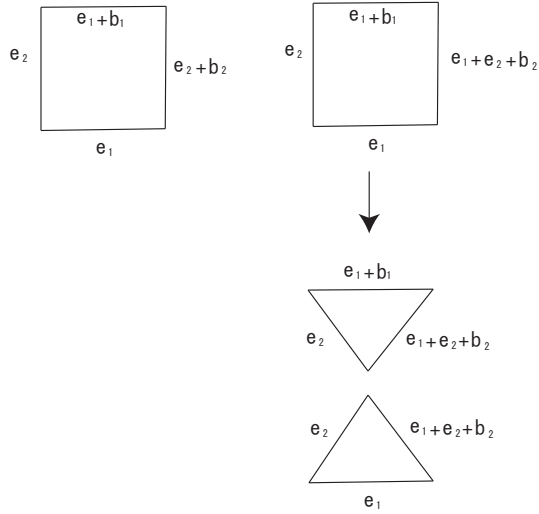}
\caption{We can not do $\sharp_{\Delta^{k-1}}^{-1}$ for the left $4$-gon; however, we can do $\sharp_{\Delta^{k-1}}^{-1}$ for the right $4$-gon.}
\label{explain3}
\end{figure}
If we can do $\sharp_{\Delta^{k-1}}$ on a $4$-gon, then we get two $3$-gons (see the right in Figure \ref{explain3}).
Consequently, we get $3$-gons and $4$-gons which we can not do $\sharp_{\Delta^{k-1}}$ any more from an $m$-gon ($m\ge 3$).

It is easy to see that such $3$-gons and $4$-gons have the characteristic functions illustrated in Figure \ref{explain4}.
\begin{figure}[h]
\includegraphics[width=250pt,clip]{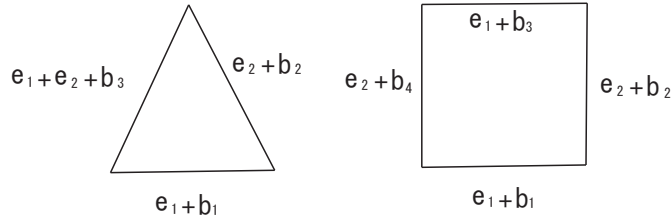}
\caption{Characteristic functions on $3$-gons $\Delta^{2}$ and $4$-gons $I^{2}$ in the final step.}
\label{explain4}
\end{figure}
Therefore, we finally need to analyze these are the same with the projective characteristic functions.
Let us recall the (ordinary) characteristic functions on two polytopes $\Delta^{2}\times \Delta^{k-1}$ and $I^{2}\times \Delta^{k-1}$
which corresponding to the above two polytopes with projective characteristic functions (see Section \ref{sect4}).
They are
\begin{eqnarray*}
A=\left(
\begin{array}{ccccc}
1 & 0 & 0  \cdots  0 & 1 & 0  \\
0 & 1 & 0  \cdots  0 & 1 & 0  \\
b_{1} & b_{2} & I_{k-1}  & b_{3} & \textbf{1}
\end{array}
\right), \quad
B=\left(
\begin{array}{cccccc}
1 & 0 & 0  \cdots  0 & 1 & 0 & 0  \\
0 & 1 & 0  \cdots  0 & 0 & 1 & 0  \\
b_{1} & b_{2} & I_{k-1}  & b_{3} & b_{4} & \textbf{1}
\end{array}
\right),
\end{eqnarray*}
respectively.
As is well known, two $n$-dimensional small covers are weakly equivariantly homeomorphic
if and only if the corresponding characteristic matrices are the same up to the left multiplication of some $X\in GL(n;\Z/2\Z)$.
Let $X$ be the following $(k+1)\times (k+1)$-matrix:
\begin{eqnarray*}
X=\left(
\begin{array}{ccc}
1 & 0 & 0  \cdots  0  \\
0 & 1 & 0  \cdots  0 \\
b_{1} & b_{2} & I_{k-1}
\end{array}
\right).
\end{eqnarray*}
Multiplying $X$ to the $A$ and $B$ above, then we have that
\begin{eqnarray*}
XA=\left(
\begin{array}{ccccc}
1 & 0 & 0  \cdots  0 & 1 & 0  \\
0 & 1 & 0  \cdots  0 & 1 & 0  \\
\textbf{0} & \textbf{0} & I_{k-1}  & b_{1}+b_{2}+b_{3} & \textbf{1}
\end{array}
\right), \quad
XB=\left(
\begin{array}{cccccc}
1 & 0 & 0  \cdots  0 & 1 & 0 & 0  \\
0 & 1 & 0  \cdots  0 & 0 & 1 & 0  \\
\textbf{0} & \textbf{0} & I_{k-1}  & b_{1}+b_{3} & b_{2}+b_{4} & \textbf{1}
\end{array}
\right).
\end{eqnarray*}
This implies that the small covers with characteristic functions $A$ and $B$ are weak equivariantly homeomorphic to the
projective bundles over small covers defined by $(\Delta^{2},\lambda_{\Delta^{2}})$ and $(I^{2},\lambda_{I^{2}})$ such that
the projective characteristic functions $\lambda_{\Delta^{2}}$ and $\lambda_{I^{2}}$ are induced from the above $XA$ and $XB$.
Note that $(\Delta^{2}, \lambda_{\Delta^{2}})$ corresponds to the projective bundle over $\R P(2)$.
On the other hand,
$(I^{2},\lambda_{I^{2}})$ corresponds to the projective bundle over $T^{2}$.
Therefore, by using the finite times $\sharp_{\Delta^{k-1}}^{-1}$,
the projective bundle over $M$ can be decomposed into projective bundles over $\R P(2)$ and $T^{2}$.
It follows from the converse of this argument that we establish the statement of this theorem.
\end{proof}

By using Remark \ref{relation with the connected sum} and Theorem \ref{construction},
we have the following well-known fact.
%%%%%%%%%%%%%%%%%%%%%%%%%%%%%%%%%%%%%%%%%%%%%%%%%%%%%%%%%%%%%%%%%%%%%%%%%%%
%Corollary 5.3
%%%%%%%%%%%%%%%%%%%%%%%%%%%%%%%%%%%%%%%%%%%%%%%%%%%%%%%%%%%%%%%%%%%%%%%%%%%
\begin{corollary}
\label{2-dim}
Let $M^{2}$ be a $2$-dimensional small cover.
Then $M^{2}$ is equivariantly homeomorphic to an equivariant connected sum of finite $\R P(2)$'s and $T^{2}$'s with standard $\Z_{2}^{2}$-actions.
\end{corollary}

Here, in Corollary \ref{2-dim}, the standard $\Z_{2}^{n}$-action on $\R P^{n}$ is defined by
\begin{eqnarray*}
(t_{1},\ldots, t_{n})\cdot [r_{0}:r_{1}\cdots : r_{n}]\mapsto [r_{0}:t_{1}r_{1}\cdots : t_{n}r_{n}]
\end{eqnarray*}
where $(t_{1},\ldots, t_{n})\in \Z_{2}^{n}$ and $[r_{0}:r_{1}\cdots : r_{n}]\in \R P^{n}$,
and we regard $T^{2}$ as the product of two $\R P^{1}$ with the standard $\Z_{2}$-actions.

%%%%%%%%%%%%%%%%%%%%%%%%%%%%%%%%%%%%%%%%%%%%%%%%%%%%%%%%%%%%%%%%%%%%%%%%%%%
%Remark 5.4
%%%%%%%%%%%%%%%%%%%%%%%%%%%%%%%%%%%%%%%%%%%%%%%%%%%%%%%%%%%%%%%%%%%%%%%%%%%
\begin{remark}
In Theorem \ref{construction},
in order to construct the projective bundle from the basic projective bundles, i.e., $P(\kappa)$ and $P(\zeta)$,
we do not need to use the operation $\sharp_{\Delta^{k-1}}^{-1}$, i.e., it is enough to use the projective fibre sum $\sharp_{\Delta^{k-1}}$ only.
\end{remark}

%%%%%%%%%%%%%%%%%%%%%%%%%%%%%%%%%%%%%%%%%%%%%%%%%%%%%%%%%%%%%%%%%%%%%%%%%%%
%Remark 5.5
%%%%%%%%%%%%%%%%%%%%%%%%%%%%%%%%%%%%%%%%%%%%%%%%%%%%%%%%%%%%%%%%%%%%%%%%%%%
%\begin{remark}
%One can easily show the above type theorem for the projective bundles over $4$-dimensional quasitoric manifolds by applying the Orlick-Raymond's method in \cite{OR}.
%Using the recent Nishimura's results \cite{Nishimura} and introducing the new operation which is the projective bundle analogue of the equivariant surgery (also see \cite{Izmestiev, Kuroki, LY}),
%we may also prove that the above type theorem for the projective bundles over $3$-dimensional small covers.
%\end{remark}

%%%%%%%%%%%%%%%%%%%%%%%%%%%%%%%%%%%%%%%%%%%%%%%%%%%%%%%%%%%%%%%%%%%%%%%%%%%
%Remark 5.5
%%%%%%%%%%%%%%%%%%%%%%%%%%%%%%%%%%%%%%%%%%%%%%%%%%%%%%%%%%%%%%%%%%%%%%%%%%%
\begin{remark}
Recall that the real line bundle over $\R P^{1}\cong S^{1}$ (i.e., $1$-dimensional small cover) can be written as
the quotient space
$S^{1}\times_{\Z_{2}}\R_{\alpha}$ by the free $\Z_{2}$ action on $S^{1}$ and the representation $\alpha:\Z_{2}\to \Z_{2}\in (\Z/2\Z)$ (i.e., trivial or non-trivial) and that all vector bundles over $S^{1}$ can be split into line bundles.
Therefore, all projectivization of vector bundles over $S^{1}$ is homeomorphic to
\[
S^{1}\times_{\Z_{2}}P(\R_{\alpha_{1}}\oplus\cdots\oplus\R_{\alpha_{k-1}}\oplus\underline{\R})
\]
for some vector $(\alpha_{1},\ldots,\alpha_{k-1})\in (\Z/2\Z)^{k-1}$, where $S^{1}\times_{\Z_{2}}\underline{\R}=S^{1}\times \R$
(i.e., the trivial bundle).
\end{remark}

%%%%%%%%%%%%%%%%%%%%%%%%%%%%%%%%%%%%%%%%%%%%%%%%%%%%%%%%%%%%%%%%%%%%%%%%%%%
%Section 6
%%%%%%%%%%%%%%%%%%%%%%%%%%%%%%%%%%%%%%%%%%%%%%%%%%%%%%%%%%%%%%%%%%%%%%%%%%%
\section{Topological classification of projective bundles over $\R P(2)$ and $T^{2}$}
\label{sect6}

In this section, we give the topological classification of $P(\kappa)$ and $P(\zeta)$ appeared in Theorem \ref{construction}, i.e.,
the classification of the topological types of projective bundles over $\R P(2)$ and $T^{2}$.
As we assumed before, all vector bundles in this section are split into the Whitney sum of line bundles.

%%%%%%%%%%%%%%%%%%%%%%%%%%%%%%%%%%%%%%%%%%%%%%%%%%%%%%%%%%%%%%%%%%%%%%%%%%%
%Section 6.1
%%%%%%%%%%%%%%%%%%%%%%%%%%%%%%%%%%%%%%%%%%%%%%%%%%%%%%%%%%%%%%%%%%%%%%%%%%%
\subsection{Topological classification of projective bundles over $\R P(2)$}
\label{sect6.1}

The classification of projective bundles over $\R P(2)$ is known by Masuda's paper \cite{M}.
Due to \cite{M},
we have $q\equiv q'$ or $k-q'$ (${\rm mod}\ 4$) if and only if
$S^{2}\times_{\Z_{2}}P(q\gamma\oplus (k-q)\epsilon)\cong S^{2}\times_{\Z_{2}}P(q'\gamma\oplus (k-q')\epsilon)$,
where $\Z_{2}$ acts on $S^{2}$ diagonally and $\gamma$ represents the tautological line bundle over $\R P(2)$, i..e,
$E(\gamma)\equiv S^{2}\times_{\Z_{2}} \R$ such that $\Z_{2}$ acts on $\R$ standardly.
Note that a line bundle over $\R P(2)$ is $\gamma$ or the trivial line bundle $\epsilon$.
%Therefore, by Proposition \ref{basic}, we see that the projective bundles over $\R P(2)$ are only these two cases.
By using this fact (and comparing the cohomology rings), we can easily check the following proposition:
%%%%%%%%%%%%%%%%%%%%%%%%%%%%%%%%%%%%%%%%%%%%%%%%%%%%%%%%%%%%%%%%%%%%%%%%%%%
%Proposition 6.1
%%%%%%%%%%%%%%%%%%%%%%%%%%%%%%%%%%%%%%%%%%%%%%%%%%%%%%%%%%%%%%%%%%%%%%%%%%%
\begin{proposition}
\label{2-proj}
Let $P(\kappa)\cong P(q\gamma\oplus(k-q)\epsilon)$ be a projective bundle over $\R P(2)$.
Then, it is homeomorphic to one of the following distinct manifolds.
\begin{enumerate}
\item The case $k\equiv 0$ $({\rm mod}\ 4)$:
\begin{enumerate}
\item if $q\equiv 0$ $({\rm mod}\ 4)$, then $P(q\gamma\oplus(k-q)\epsilon)\cong \R P(2)\times \R P(k-1)$;
\item if $q\equiv 1,\ 3$ $({\rm mod}\ 4)$, then $P(q\gamma\oplus(k-q)\epsilon)\cong S^{2}\times_{\Z_{2}}P(\gamma\oplus (k-1)\epsilon)$;
\item if $q\equiv 2$ $({\rm mod}\ 4)$, then $P(q\gamma\oplus(k-q)\epsilon)\cong S^{2}\times_{\Z_{2}}P(2\gamma\oplus (k-2)\epsilon)$.
\end{enumerate}
\item The case $k\equiv 1$ $({\rm mod}\ 4)$:
\begin{enumerate}
\item if $q\equiv 0,\ 1$ $({\rm mod}\ 4)$, then $P(q\gamma\oplus(k-q)\epsilon)\cong \R P(2)\times \R P(k-1)$;
\item if $q\equiv 2,\ 3$ $({\rm mod}\ 4)$, then $P(q\gamma\oplus(k-q)\epsilon)\cong S^{2}\times_{\Z_{2}}P(2\gamma\oplus (k-2)\epsilon)$.
\end{enumerate}
\item The case $k\equiv 2$ $({\rm mod}\ 4)$:
\begin{enumerate}
\item if $q\equiv 0,\ 2$ $({\rm mod}\ 4)$, then $P(q\gamma\oplus(k-q)\epsilon)\cong \R P(2)\times \R P(k-1)$;
\item if $q\equiv 1$ $({\rm mod}\ 4)$, then $P(q\gamma\oplus(k-q)\epsilon)\cong S^{2}\times_{\Z_{2}}P(\gamma\oplus (k-1)\epsilon)$;
\item if $q\equiv 3$ $({\rm mod}\ 4)$, then $P(q\gamma\oplus(k-q)\epsilon)\cong S^{2}\times_{\Z_{2}}P(3\gamma\oplus (k-3)\epsilon)$.
\end{enumerate}
\item The case $k\equiv 3$ $({\rm mod}\ 4)$:
\begin{enumerate}
\item if $q\equiv 0,\ 3$ $({\rm mod}\ 4)$, then $P(q\gamma\oplus(k-q)\epsilon)\cong \R P(2)\times \R P(k-1)$;
\item if $q\equiv 1,\ 2$ $({\rm mod}\ 4)$, then $P(q\gamma\oplus(k-q)\epsilon)\cong S^{2}\times_{\Z_{2}}P(\gamma\oplus (k-1)\epsilon)$.
\end{enumerate}
\end{enumerate}
\end{proposition}
Note that the moment-angle manifold over $\R P(2)$ is $S^{2}$.

%%%%%%%%%%%%%%%%%%%%%%%%%%%%%%%%%%%%%%%%%%%%%%%%%%%%%%%%%%%%%%%%%%%%%%%%%%%
%Section 6.2
%%%%%%%%%%%%%%%%%%%%%%%%%%%%%%%%%%%%%%%%%%%%%%%%%%%%%%%%%%%%%%%%%%%%%%%%%%%
\subsection{Topological classification of projective bundles over $T^{2}$}
\label{sect6.2}

Next we classify projective bundles over $T^{2}$.
Let $\gamma_{i}$ be the pull back of the canonical line bundle over $S^{1}$ by the $i$th factor projection $\pi_{i}:T^{2}\to S^{1}$ ($i=1,\ 2$).
We can easily show that line bundles over $T^{2}$ is completely determined by its $1$st Stiefel-Whitney classes via
$[T^{2},\ B\Z_{2}]\simeq H^{1}(T^{2};\ \Z/2\Z)\simeq (\Z/2\Z)^{2}$.
Therefore, all of the line bundles over $T^{2}$ are $\epsilon$, $\gamma_{1}$, $\gamma_{2}$ and $\gamma_{1}\otimes\gamma_{2}$.
By the definition of $\gamma_{i}$, we can easily show that
\begin{eqnarray}
\label{tech1}
\gamma_{i}\oplus \gamma_{i}=\pi_{i}^{*}(\gamma\oplus\gamma)=\pi_{i}^{*}(2\epsilon)=2\epsilon.
\end{eqnarray}
Therefore, we also have
\begin{eqnarray}
\label{tech2}
(\gamma_{1}\otimes \gamma_{2})\oplus (\gamma_{1}\otimes \gamma_{2})=\gamma_{1}\otimes (\gamma_{2}\oplus \gamma_{2})
=\gamma_{1}\otimes 2\epsilon=\gamma_{1}\oplus\gamma_{1}=2\epsilon.
\end{eqnarray}
Let $\zeta$ be a $k$-dimensional vector bundle ($k\ge 2$).
Because $\dim T^{2}=2$, if $k\ge 2$ then $\zeta$ is in the stable range.
Therefore,
we have that
\[
\zeta\equiv \zeta^{2}\oplus (k-2)\epsilon,
\]
where $\zeta^{2}$ is a $2$-dimensional vector bundle over $T^{2}$.
Hence, if $\zeta$ is a Whitney sum of $k$ line bundles then $\zeta$ is isomorphic to one of the followings by computing the Stiefel-Whitney class:
\begin{eqnarray*}
k\epsilon; \\
\gamma_{1}\oplus (k-1)\epsilon; \\
\gamma_{2}\oplus (k-1)\epsilon; \\
(\gamma_{1}\otimes \gamma_{2})\oplus (k-1)\epsilon; \\
\gamma_{1}\oplus \gamma_{2}\oplus (k-2)\epsilon; \\
\gamma_{1}\oplus (\gamma_{1}\otimes\gamma_{2})\oplus (k-2)\epsilon; \\
\gamma_{2}\oplus (\gamma_{1}\otimes\gamma_{2})\oplus (k-2)\epsilon.
\end{eqnarray*}
By using this classification, we can prove the following proposition.

%%%%%%%%%%%%%%%%%%%%%%%%%%%%%%%%%%%%%%%%%%%%%%%%%%%%%%%%%%%%%%%%%%%%%%%%%%%
%Proposition 6.2
%%%%%%%%%%%%%%%%%%%%%%%%%%%%%%%%%%%%%%%%%%%%%%%%%%%%%%%%%%%%%%%%%%%%%%%%%%%
\begin{proposition}
\label{2-torus}
Let $P(\zeta)$ be a projective bundle over $T^{2}$.
Then it is homeomorphic to one of the following manifolds:
\begin{enumerate}
\item The trivial bundle $T^{2}\times \R P(k-1)$;
\item The non-trivial bundle of type $T^{2}\times_{\Z_{2}^{2}}P(\R_{\rho_{1}}
\oplus \R_{\rho_{2}}\oplus \underline{\R}^{k-2})$;
\item The non-trivial bundle of type $T^{2}\times_{\Z_{2}^{2}}P(\R_{\rho_{1}}\oplus \underline{\R}^{k-1})\cong T^{2}\times_{\Z_{2}^{2}}P(\R_{\rho_{2}}\oplus \underline{\R}^{k-1})$,
\end{enumerate}
where $\rho_{i}:\Z_{2}^{2}\to \Z_{2}$ is the $i$th projection and $\underline{\R}$ is the trivial representation space.

When $k>2$, each manifold above has different topological types; however, 
when $k=2$, both of two non-trivial bundles above are isomorphic to the non-trivial bundle $T^{2}\times_{\Z_{2}^{2}}P(\R_{\rho_{1}}\oplus \underline{\R})$.
\end{proposition}
\begin{proof}
Recall that $P(\zeta\otimes \gamma)=P(\zeta)$ for all line bundles $\gamma$.
Therefore, by using the classification of vector bundles over $T^{2}$ just before this proposition and the relations \eqref{tech1}, \eqref{tech2}, 
it is easy to check that
the topological types of $P(\zeta)$ are one of the followings.
\begin{enumerate}
\item The case $k\equiv 0$ $({\rm mod}\ 2)$:
\begin{enumerate}
\item $P(k\epsilon)\cong T^{2}\times \R P(k-1)$;
\item $P((\gamma_{1}\otimes\gamma_{2})\oplus(k-1)\epsilon)\cong P(\gamma_{1}\oplus\gamma_{2}\oplus(k-2)\epsilon)\cong T^{2}\times_{\Z_{2}^{2}}P(\R_{\rho_{1}}
\oplus \R_{\rho_{2}}\oplus \underline{\R}^{k-2})$;
\item $P(\gamma_{1}\oplus(k-1)\epsilon)\cong P((\gamma_{1}\otimes\gamma_{2})\oplus\gamma_{2}\oplus(k-2)\epsilon)\cong T^{2}\times_{\Z_{2}^{2}}P(\R_{\rho_{1}}
\oplus \underline{\R}^{k-1})$;
\item $P(\gamma_{2}\oplus(k-1)\epsilon)\cong P((\gamma_{1}\otimes\gamma_{2})\oplus\gamma_{1}\oplus(k-2)\epsilon)\cong T^{2}\times_{\Z_{2}^{2}}P(\R_{\rho_{2}}
\oplus \underline{\R}^{k-1})$;
\end{enumerate}
\item The case $k\equiv 1$ $({\rm mod}\ 2)$:
\begin{enumerate}
\item $P(k\epsilon)\cong T^{2}\times \R P(k-1)$;
\item $P((\gamma_{1}\otimes \gamma_{2})\oplus (k-1)\epsilon)\cong P(\gamma_{1}\oplus (k-1)\epsilon)\cong P(\gamma_{2}\oplus (k-1)\epsilon)\cong
T^{2}\times_{\Z_{2}^{2}}P(\R_{\rho_{1}}\oplus \underline{\R}^{k-1})$;
\item $P((\gamma_{1}\oplus \gamma_{2})\oplus (k-2)\epsilon)\cong P((\gamma_{1}\otimes \gamma_{2})\oplus\gamma_{1}\oplus (k-2)\epsilon)
\cong P((\gamma_{1}\otimes \gamma_{2})\oplus\gamma_{2}\oplus (k-2)\epsilon)\cong
T^{2}\times_{\Z_{2}^{2}}P(\R_{\rho_{1}}\oplus\R_{\rho_{2}}\oplus \underline{\R}^{k-2})$.
\end{enumerate}
\end{enumerate}
By using the Borel-Hirzebruch formula, we have the cohomology ring of $P(\zeta)$ as the following list:
\begin{center}
\begin{tabular}{|c|c|}
\hline
$P(\zeta)$ & $H^{*}(\cdot)$  \\
\hline
\hline
$T^{2}\times \R P(k-1)$ & $\Z/2\Z[x,\ y,\ z]/\langle x^{2},\ y^{2},\ z^{k}\rangle$  \\[2pt]
\hline
$T^{2}\times_{\Z_{2}^{2}}P(\R_{\rho_{1}}\oplus \R_{\rho_{2}}\oplus \underline{\R}^{k-2})$ & $\Z/2\Z[x,\ y,\ z]/\langle x^{2},\ y^{2},\ z^{k}+z^{k-1}(x+y)+z^{k-2} x y\rangle$  \\[2pt]
\hline
$T^{2}\times_{\Z_{2}^{2}}P(\R_{\rho_{1}}\oplus \underline{\R}^{k-1})$ & $\Z/2\Z[x,\ y,\ z]/\langle x^{2},\ y^{2},\ z^{k}+z^{k-1}x\rangle$  \\[2pt]
\hline
$T^{2}\times_{\Z_{2}^{2}}P(\R_{\rho_{2}}\oplus \underline{\R}^{k-1})$ & $\Z/2\Z[x,\ y,\ z]/\langle x^{2},\ y^{2},\ z^{k}+z^{k-1}y\rangle$ \\[2pt]
\hline
\end{tabular}
\end{center}
for $\deg x=\deg y=\deg z=1$.
This implies that the bundles as above are not homeomorphic to each other except (1)-(c) and (1)-(d) when $k>2$.
It is easy to check that
\begin{eqnarray*}
T^{2}\times_{\Z_{2}^{2}}P(\R_{\rho_{1}}\oplus \underline{\R}^{k-1})\cong
S^{1}\times (S^{1}\times_{\Z_{2}}P(\R_{\rho}\oplus \underline{\R}^{k-1})) \cong
T^{2}\times_{\Z_{2}^{2}}P(\R_{\rho_{2}}\oplus \underline{\R}^{k-1}),
\end{eqnarray*}
where $S^{1}\times_{\Z_{2}}\R_{\rho}$ is the canonical line bundle over $\R P(1)$.
This establishes the statement except the case when $k=2$.

When $k=2$, we have that
\begin{eqnarray*}
T^{2}\times_{\Z_{2}^{2}}P(\R_{\rho_{1}}\oplus \R_{\rho_{2}})\cong
T^{2}\times_{\Z_{2}^{2}}P(\R_{\rho'}\oplus \underline{\R}),
\end{eqnarray*}
where $\rho':T^{2}\to S^{1}$ is the representation $(t_{1},t_{2})\mapsto t_{1}t_{2}$.
By using the kernel of this representation $\Delta=\{(t,t^{-1})\ |\ t\in S^{1}\}$,
we also have the following homeomorphism:
\begin{eqnarray*}
T^{2}\times_{\Z_{2}^{2}}P(\R_{\rho'}\oplus \underline{\R})\cong
\Delta\times (S^{1}\times_{\Z_{2}} P(\R_{\rho}\oplus \underline{\R}))\cong
S^{1}\times (S^{1}\times_{\Z_{2}}P(\R_{\rho}\oplus \underline{\R})),
\end{eqnarray*}
where $S^{1}\times_{\Z_{2}}\R_{\rho}$ is the canonical line bundle over $\R P(1)$.
Similarly, we have the following homeomorphisms:
\[
T^{2}\times_{\Z_{2}^{2}}P(\R_{\rho_{1}}\oplus \underline{\R})\cong
S^{1}\times (S^{1}\times_{\Z_{2}}P(\R_{\rho}\oplus \underline{\R})) \cong
T^{2}\times_{\Z_{2}^{2}}P(\R_{\rho_{2}}\oplus \underline{\R}).
\]
This also establishes the case when $k=2$.
\end{proof}

Note that the moment-angle manifold over $T^{2}$ is $T^{2}$ itself.

It also follows from the proof of Proposition \ref{2-torus} that the following corollary holds.

%%%%%%%%%%%%%%%%%%%%%%%%%%%%%%%%%%%%%%%%%%%%%%%%%%%%%%%%%%%%%%%%%%%%%%%%%%%
%Corollary 6.3
%%%%%%%%%%%%%%%%%%%%%%%%%%%%%%%%%%%%%%%%%%%%%%%%%%%%%%%%%%%%%%%%%%%%%%%%%%%
\begin{corollary}
\label{cohomological rigidity}
Let $\mathcal{P}(T^{2})$ be the set of all projective bundles over $T^{2}$ and $P(\zeta_{1}),\ P(\zeta_{2})\in \mathcal{P}(T^{2})$.
Then, $H^{*}(P(\zeta_{1}))\simeq H^{*}(P(\zeta_{2}))$ if and only if $P(\zeta_{1})\cong P(\zeta_{2})$ (homeomorphic), i.e.,
$\mathcal{P}(T^{2})$ satisfies cohomological rigidity.
\end{corollary}

%%%%%%%%%%%%%%%%%%%%%%%%%%%%%%%%%%%%%%%%%%%%%%%%%%%%%%%%%%%%%%%%%%%%%%%%%%%
%Remark 6.4
%%%%%%%%%%%%%%%%%%%%%%%%%%%%%%%%%%%%%%%%%%%%%%%%%%%%%%%%%%%%%%%%%%%%%%%%%%%
\begin{remark}
\label{remark from Masuda}
Let $\mathcal{P}(\R P(2))$ be the set of all projective bundles over $\R P(2)$.
Due to \cite[Theorem 3.3]{M},
$\mathcal{P}(\R P(2))$ also satisfies cohomological rigidity.
\end{remark}

%%%%%%%%%%%%%%%%%%%%%%%%%%%%%%%%%%%%%%%%%%%%%%%%%%%%%%%%%%%%%%%%%%%%%%%%%%%
%Section 7
%%%%%%%%%%%%%%%%%%%%%%%%%%%%%%%%%%%%%%%%%%%%%%%%%%%%%%%%%%%%%%%%%%%%%%%%%%%
\section{Topological triviality of some projective bundles over real projective spaces}
\label{sect7}

Let $\tau_{\R P^{n}}$ be the tangent bundle over $\R P^{n}$.

The following relation is well-known:
\begin{eqnarray}
\label{basic}
\epsilon\oplus \tau_{\R P^{n}}\equiv (n+1)\gamma,
\end{eqnarray}
where $\epsilon$ is the trivial line bundle over $\R P^{n}$ and $(n+1)\gamma$ represents the $(n+1)$-times Whitney sum of
the tautological line bundle $\gamma$.
This relation \eqref{basic} shows that
\[
P(\epsilon\oplus\tau_{\R P^{n}})\cong P((n+1)\gamma)\cong P((n+1)(\gamma\otimes \gamma))\cong P((n+1)\epsilon)\cong \R P^{n}\times \R P^{n}.
\]
Therefore, the projectivization of $\epsilon\oplus \tau_{\R P^{n}}$ always admits the trivial topology.
Since the line bundle over $\R P^{n}$ is just the trivial bundle $\epsilon$ or the tautological line bundle $\gamma$,
it is natural to ask this question to the projectivization of $\gamma\oplus \tau_{\R P^{n}}$. 
(Note that this might not be a small cover).
In this section, we answer the following question asked by Richard Montgomery motivated
from his interest of the study of singularities \cite{CM}:

\begin{problem}[R. Montgomery]
When does $P(\gamma\oplus\tau_{\R P^{n}})$ have the trivial topology?
In other wards, when is $P(\gamma\oplus\tau_{\R P^{n}})$ diffeomorphic (or homeomorphic) to $\R P^{n}\times \R P^{n}$?
\end{problem}

Namely we prove the following theorem:
%%%%%%%%%%%%%%%%%%%%%%%%%%%%%%%%%%%%%%%%%%%%%%%%%%%%%%%%%
% Theorem 7.1
%%%%%%%%%%%%%%%%%%%%%%%%%%%%%%%%%%%%%%%%%%%%%%%%%%%%%%%%%
\begin{theorem}
\label{main-tri}
The projectivization $P(\gamma\oplus\tau_{\R P^{n}})$ is diffeomorphic to $\R P^{n}\times \R P^{n}$ if and only if
$n=0,\ 2$ or $6$.
\end{theorem}

In order to prove Theorem \ref{main-tri},
we first show when cohomology ring of $P(\gamma\oplus\tau_{\R P^{n}})$ is isomorphic to that of $\R P^{n}\times \R P^{n}$: %the $\Z_{2}$-cohomologically triviality is the first criteria of topological triviality.
%%%%%%%%%%%%%%%%%%%%%%%%%%%%%%%%%%%%%%%%%%%%%%%%%%%%%%%%%
% Lemma 7.2
%%%%%%%%%%%%%%%%%%%%%%%%%%%%%%%%%%%%%%%%%%%%%%%%%%%%%%%%%
\begin{lemma}
\label{cohomological triviality}
The $\Z/2\Z$-cohomology ring of $P(\gamma\oplus\tau_{\R P^{n}})$ is isomorphic to that of $\R P^{n}\times \R P^{n}$ if and only if
$n+2=2^{r}$ for some $r\in \N$.
\end{lemma}
\begin{proof}
Because of \eqref{basic}, $\gamma\oplus\tau_{\R P^{n}}\oplus\epsilon=(n+2)\gamma$.
Therefore, we have that
\[
\omega(\gamma\oplus\tau_{\R P^{n}})=(1+x)^{n+2}\equiv \sum_{i=0}^{n}\left(
\begin{array}{c}
n+2 \\
i
\end{array}
\right)
x^{i}
\]
for $x\in H^{1}(\R P^{n})$.
Together with the Borel-Hirzebruch formula,
we see that the cohomology ring of $P(\gamma\oplus\tau_{\R P^{n}})$ is as follows:
\begin{eqnarray*}
H^{*}(P(\gamma\oplus\tau_{\R P^{n}}))\simeq \Z/2\Z[x,y]/\langle x^{n+1}, Y\rangle.
\end{eqnarray*}
Here, 
\begin{eqnarray*}
Y=\sum_{i=0}^{n}\left(
\begin{array}{c}
n+2 \\
i
\end{array}
\right)
y^{n+1-i}
x^{i}
\end{eqnarray*}
%\underline{By using the similar arguments in \cite{MS}}, 
Note that $n+2=2^{r}$ if and only if $Y=y^{n+1}$ (e.g. see \cite[Corollary 4.6]{MS}).
Therefore, if $n+2=2^{r}$ then $Y=y^{n+1}$ and the cohomology ring is isomorphic to $H^{*}(\R P^{n}\times \R P^{n})$.
On the other hand, if the cohomology ring is isomorphic to $H^{*}(\R P^{n}\times \R P^{n})$, 
then it is easy to check that $Y$ must be $y^{n+1}$ or $(x+y)^{n+1}$.
However, if $Y=(x+y)^{n+1}$ then
\begin{eqnarray*}
Y=\sum_{i=0}^{n}\left(
\begin{array}{c}
n+2 \\
i
\end{array}
\right)
y^{n+1-i}
x^{i}
=
\sum_{i=0}^{n}\left(
\begin{array}{c}
n+1 \\
i
\end{array}
\right)
y^{n+1-i}
x^{i}.
\end{eqnarray*}
This gives a contradiction.
Therefore, $Y=y^{n+1}$ and $n+2=2^{r}$.
This establishes the statement of this lemma.
\end{proof}

Lemma \ref{cohomological triviality} tells us that if
$n+2\not=2^{r}$ for all $r\in \N$ then
$P(\gamma\oplus\tau_{\R P^{n}})$ is not homeomorphic to $\R P^{n}\times \R P^{n}$.

Assume $n+2=2^{r}$ for some $r\in \N$.
If $r=1$ then $n=0$, so this case is the trivial case.
We may assume $r\ge 2$.

%%%%%%%%%%%%%%%%%%%%%%%%%%%%%%%%%%%%%%%%%%%%%%%%%%%%%%%%%%%%%%%%%%%%%%%%%%%
%Section 7.1
%%%%%%%%%%%%%%%%%%%%%%%%%%%%%%%%%%%%%%%%%%%%%%%%%%%%%%%%%%%%%%%%%%%%%%%%%%%
\subsection{``if'' part of Theorem \ref{main-tri}}
\label{sect7.1}

We next show when $\gamma\oplus\tau_{\R P^{n}}$ is the trivial bundle.
To show this, we need the fact about the stable KO group in \cite{A} (also see \cite{M}).
Before we state Lemma \ref{stable-KO}, we need to prepare some notation.
Let $k(2^{r}-1)=\#\{s\in \N\ |\ 0<s\le 2^{r}-2, s\equiv 0,1,2,4\mod 8 \}$.
For example, $k(3)=2$ when $r=2$, $k(7)=3$ when $r=3$, $k(15)=7$ when $r=4$, $k(31)=15$ when $r=5$, e.t.c.
%& &r=2 \Rightarrow k(3)=2; \\ %1,2
%& &r=3 \Rightarrow k(7)=3; \\ %1,2,4,
%& &r=4 \Rightarrow k(15)=7; \\ %1,2,4, 8,9,10,12
%& &r=5 \Rightarrow k(31)=15; \\ %1,2,4, 8,9,10,12, 16,17,18,20, 24,25,26,28,
%& &r=6 \Rightarrow k(63)=31; \\ %1,2,4, 8,9,10,12, 16,17,18,20, 24,25,26,28, 32,33,34,36, 40,41,42,44, 48,49,50,52, 56,57,58,60
%& &r=7 \Rightarrow k(127)=63; \\ %1,2,4, 8,9,10,12, 16,17,18,20, 24,25,26,28, 32,33,34,36, 40,41,42,44, 48,49,50,52, 56,57,58,60
%64,65,66,68, 72,73,74,76, 80,81,82,84, 88,89,90,92, 96,97,98,100, 104,105,106,108, 112,113,114,116, 120,121,122,124,
We have the following lemma:
%%%%%%%%%%%%%%%%%%%%%%%%%%%%%%%%%%%%%%%%%%%%%%%%%%%%%%%%%
% Lemma 7.3
%%%%%%%%%%%%%%%%%%%%%%%%%%%%%%%%%%%%%%%%%%%%%%%%%%%%%%%%%
\begin{lemma}
\label{number-theoretical-condition}
If $r=2$ or $3$, then $k(2^{r}-1)=r$.
If $r\ge 4$, then $k(2^{r}-1)=2^{r-1}-1$.
\end{lemma}
\begin{proof}
The first statement is easy.
The 2nd statement is proved by induction.
When $r=4$, then $k(15)=7$.
Assume the statement is true until $r-1$, i.e., $k(2^{r-1}-1)=2^{r-2}-1$.
Because of the definition of $k(2^{r}-1)$, the number of $s$ such that $0<s\le 2^{r}-2$ and $s\equiv 0,1,2,4$ (${\rm mod}\ 8$) is
\[
k(2^{r}-1)=(2^{r-2}-1)+4 \cdot 2^{r-4}=2^{r-1}-1.
\]
This establishes the statement.
\end{proof}

Together with the stable KO group of real projective space proved in \cite{A}, we have the following lemma:
%%%%%%%%%%%%%%%%%%%%%%%%%%%%%%%%%%%%%%%%%%%%%%%%%%%%%%%%%
% Lemma 7.4
%%%%%%%%%%%%%%%%%%%%%%%%%%%%%%%%%%%%%%%%%%%%%%%%%%%%%%%%%
\begin{lemma}
\label{stable-KO}
When $r=2,3$, $\widetilde{KO}(\R P^{2^{r}-2})$ is a cyclic group generated by $\gamma-\epsilon$ with order $4, 8$, respectively.
When $r\ge 4$,
$\widetilde{KO}(\R P^{2^{r}-2})$ is a cyclic group generated by $\gamma-\epsilon$ with order $2^{(2^{r-1}-1)}$.
\end{lemma}
Note that $\gamma\oplus\tau_{\R P^{n}}$ is in the stable range, i.e., the dimension of fibre is strictly greater than $n$.
Because of the stable range theorem (i.e., for vector bundles $\kappa$ and $\eta$ in the stable range, $\kappa\oplus\epsilon^{a}\equiv\eta\oplus\epsilon^{a}$ iff $\kappa\equiv\eta$, see \cite[Chapter 9]{H}),
$\gamma\oplus\tau_{\R P^{n}}$ is the trivial bundle if and only if it is the trivial bundle in $\widetilde{KO}(\R P^{n})$.
By this fact, we have the following proposition:
%%%%%%%%%%%%%%%%%%%%%%%%%%%%%%%%%%%%%%%%%%%%%%%%%%%%%%%%%
% Proposition 7.5
%%%%%%%%%%%%%%%%%%%%%%%%%%%%%%%%%%%%%%%%%%%%%%%%%%%%%%%%%
\begin{lemma}
\label{n=2,6}
Assume $n=2^{r}-2$. Then
$\gamma\oplus \tau_{\R P^{n}}\equiv(n+1)\epsilon$ if and only if $n=2,6$.
\end{lemma}
\begin{proof}
By using \eqref{basic}, we have that
\begin{eqnarray}
\label{basic2}
\gamma\oplus \tau_{\R P^{n}}\oplus \epsilon\equiv 2^{r}\gamma.
\end{eqnarray}
It follows from Lemma \ref{stable-KO} that when $r\ge 4$
\begin{eqnarray}
\label{relation}
2^{(2^{r-1}-1)}\gamma \equiv 2^{(2^{r-1}-1)}\epsilon.
\end{eqnarray}
Because $r<2^{r-1}-1$, together with \eqref{basic2}, this case is not the trivial bundle.
On the other hand, when $r=2,3$, we have that
\begin{eqnarray*}
2^{(2^{r-1}-1)}\gamma=2^{r}\gamma \equiv 2^{(2^{r-1}-1)}\epsilon=2^{r}\epsilon.
\end{eqnarray*}
Therefore, by \eqref{basic2} and the stable range theorem, $\gamma\oplus\tau_{\R P^{n}}$ is the trivial bundle.
This establishes the statement.
\end{proof}
Hence, by Lemma \ref{n=2,6}, the projectivization $P(\gamma\oplus\tau_{\R P^{n}})$ has the trivial topology when $n=2,6$.
This establishes the ``if'' part of Theorem \ref{main-tri}.

%%%%%%%%%%%%%%%%%%%%%%%%%%%%%%%%%%%%%%%%%%%%%%%%%%%%%%%%%%%%%%%%%%%%%%%%%%%
%Section 7.2
%%%%%%%%%%%%%%%%%%%%%%%%%%%%%%%%%%%%%%%%%%%%%%%%%%%%%%%%%%%%%%%%%%%%%%%%%%%
\subsection{``only if'' part of Theorem \ref{main-tri}}
\label{sect7.2}

We next prove the ``only if'' part of Theorem \ref{main-tri}.
The idea of this proof is based on the idea of the proof of Theorem 3.2 in \cite{M}.
Assume that there exists the following diffeomorphism:
\[
f:P=P(\gamma\oplus\tau_{\R P^{n}})\to \R P^{n}\times \R P^{n}(=P((n+1)\epsilon))=T,
\]
and we put the projections to the 1st and 2nd factor by $\pi_{1}:P\to \R P^{n}$, $\pi_{2}:T\to \R P^{n}$, respectively.
Now $f^{*}(\tau_{T})=\tau_{P}$ in $\widetilde{KO}(P)$.
Recall the following theorem proved in \cite[Lemma 3.1]{M}:
%%%%%%%%%%%%%%%%%%%%%%%%%%%%%%%%%%%%%%%%%%%%%%%%%%%%%%%%%%%%%%%%%%%%%%%%%%%
%Lemma 7.6
%%%%%%%%%%%%%%%%%%%%%%%%%%%%%%%%%%%%%%%%%%%%%%%%%%%%%%%%%%%%%%%%%%%%%%%%%%%
\begin{lemma}
Let $E\to X$ be a real smooth vector bundle over a smooth manifold $X$.
Let $\pi:P(E)\to X$ be its projectivization and $\eta$ be the tautological real line bundle of $P(E)$.
Then the tangent bundle $\tau_{P(E)}$ of $P(E)$ with $\epsilon^{1}$ added is isomorphic to ${\rm Hom}(\eta,\pi^{*}E)\oplus \pi^{*}\tau_{X}$.
\end{lemma}

By this lemma and \eqref{basic}, 
we have that 
\begin{eqnarray*}
\tau_{P}\oplus\epsilon^{1}\oplus\epsilon^{1}\equiv {\rm Hom}(\eta_{P},\pi_{1}^{*}(\gamma\oplus\tau_{\R P^{n}}))\oplus\pi_{1}^{*}\tau_{\R P^{n}}\oplus\epsilon^{1}
\equiv {\rm Hom}(\eta_{P},\gamma_{P}\oplus \pi^{*}_{1}\tau_{\R P^{n}})\oplus (n+1)\gamma_{P}
\end{eqnarray*}
and 
\begin{eqnarray*}
\tau_{T}\oplus\epsilon^{1}\oplus\epsilon^{1}\equiv {\rm Hom}(\eta_{T},\pi_{2}^{*}((n+1)\epsilon))\oplus\pi^{*}_{2}\tau_{\R P^{n}}\oplus\epsilon^{1}
\equiv {\rm Hom}(\eta_{T},(n+1)\epsilon)\oplus (n+1)\gamma_{T},
\end{eqnarray*}
where
$\gamma_{P}=\pi^{*}_{1}\gamma$ and $\gamma_{T}=\pi_{2}^{*}\gamma$ for the tautological line bundle $\gamma$ over $\R P^{n}$,
and 
$\eta_{P}$ and $\eta_{T}$ are the tautological line bundles over $P=P(\gamma\oplus\tau_{\R P^{n}})$ and $T=P(\epsilon^{n+1})$, respectively.
Together with $f^{*}\tau_{T}=\tau_{P}$,
we have the following isomorphism:
\begin{eqnarray}
\label{hom}
f^{*}\left({\rm Hom}(\eta_{T},(n+1)\epsilon)\oplus (n+1)\gamma_{T}\right)
\equiv {\rm Hom}(\eta_{P},\gamma_{P}\oplus \pi^{*}_{1}\tau_{\R P^{n}})\oplus (n+1)\gamma_{P}
\end{eqnarray}
By the cohomology ring computed in Lemma \ref{cohomological triviality}, $f^{*}w_{1}(\gamma_{T})=w_{1}(\gamma_{P})$, i.e.,
$f^{*}\gamma_{T}=\gamma_{P}$.
Therefore, by \eqref{hom}, in $\widetilde{KO}(P)$ we have
\begin{eqnarray}
\label{hom2}
{\rm Hom}(f^{*}\eta_{T},(n+1)\epsilon)
\equiv {\rm Hom}(\eta_{P},\gamma_{P}\oplus \pi^{*}_{1}\tau_{\R P^{n}}).
\end{eqnarray}
By taking the zero section to $\tau_{\R P^{n}}$, we have the cross section $\sigma$ of $\pi_{1}:P\to \R P^{n}$.
The induced homomorphism of $\sigma^{*}:\widetilde{KO}(P)\to \widetilde{KO}(\R P^{n})$ sends this identity \eqref{hom2} to
$\widetilde{KO}(\R P^{n})$.
Because $\sigma^{*}\eta_{P}$ is the trivial bundle over $\R P^{n}$, we have that
\begin{eqnarray}
\label{hom3}
{\rm Hom}(\sigma^{*}f^{*}\eta_{T},(n+1)\epsilon)
\equiv {\rm Hom}(\epsilon,\gamma\oplus \tau_{\R P^{n}})\equiv \gamma\oplus \tau_{\R P^{n}}.
\end{eqnarray}
Now,
by the cohomology ring computed in Lemma \ref{cohomological triviality} again,
we also have the two cases $f^{*}w_{1}(\eta_{T})=w_{1}(\eta_{P})$ and $w_{1}(\gamma_{P})+w_{1}(\eta_{P})$;
these correspond to $f^{*}\eta_{T}=\eta_{P}$ and $\gamma_{P}\otimes\eta_{P}$, respectively.
If $f^{*}\eta_{T}=\eta_{P}$, then by \eqref{hom3}, we have that
\[
(n+1)\epsilon\equiv \gamma\oplus \tau_{\R P^{n}}
\]
in $\widetilde{KO}(\R P^{n})$.
By Lemma \ref{n=2,6}, such case is the only $n=2$ or $6$.
If $f^{*}\eta_{T}=\gamma_{P}\otimes\eta_{P}$, then 
\begin{eqnarray*}
\sigma^{*}f^{*}\eta_{T}=\sigma^{*}\gamma_{P}\otimes\sigma^{*}\eta_{P}\equiv \gamma\otimes \epsilon\equiv \gamma.
\end{eqnarray*}
Therefore, 
by \eqref{hom3}, we have that
$(n+1)\gamma\equiv \gamma\oplus \tau_{\R P^{n}}.$
By taking the tensor of $\gamma$, we also have that
\begin{eqnarray}
\label{rel-final}
(n+1)\epsilon\equiv \epsilon\oplus(\gamma\otimes \tau_{\R P^{n}}).
\end{eqnarray}
Because $\gamma\oplus(\gamma\otimes \tau_{\R P^{n}})\equiv (n+1)\epsilon$, the vector bundle $\gamma\otimes \tau_{\R P^{n}}$
is the normal bundle $\gamma^{\perp}$ of $\gamma$ in $(n+1)\epsilon$.
Therefore,
the Stiefel-Whitney class satisfies
\[
w(\gamma\otimes \tau_{\R P^{n}})=1+x+\cdots+ x^{n}.
\]
Hence, by \eqref{rel-final}, such case is just $n=0$.
Because this case is the trivial case, we establish the ``only if'' part.

\medskip
Finally, we ask the following general question by motivating the above fact.
\begin{problem}[topological triviality problem]
Let $\xi$ be a rank $k$ vector bundle over smooth manifold $M$.
When is its projectivization $P(\xi)$ diffeomorphic to $\R P^{k-1}\times M$?
\end{problem}

\section*{Acknowledgments}
The authors would like to
thank Professor Richard Montgomery for his interest of the first survey article \cite{Kuroki2} and giving us the problem.

\end{document}